\RequirePackage[l2tabu,orthodox]{nag}
\documentclass[11pt]{amsart}
\usepackage{mathtools,amsmath,amsthm,amssymb,amsbsy,amstext,amsopn}
\usepackage{mathabx}
\usepackage{bm}
\usepackage{aliascnt}
\usepackage{xcolor}
\usepackage{graphicx}
\usepackage{tikz-cd}
\usepackage{microtype}
\usepackage{verbatim}
\usepackage[centering,margin=1.2in]{geometry}
\usepackage[colorlinks,
            linkcolor=black!50!red,
            citecolor=blue,
            pdfpagemode=None]{hyperref}
\usepackage{tikz}
\usetikzlibrary{calc,arrows,decorations.pathmorphing,intersections,decorations.markings,positioning,shapes.misc}
\numberwithin{equation}{section}

\theoremstyle{plain}
\newaliascnt{theorem}{equation}
\newtheorem{thm}[theorem]{Theorem}
\aliascntresetthe{theorem}

\newtheorem*{thm*}{Theorem}

\newaliascnt{prop}{equation}
\newtheorem{prop}[prop]{Proposition}
\aliascntresetthe{prop}

\newaliascnt{lemma}{equation}
\newtheorem{lemma}[lemma]{Lemma}
\aliascntresetthe{lemma}

\newaliascnt{corollary}{equation}
\newtheorem{corollary}[corollary]{Corollary}
\aliascntresetthe{corollary}

\theoremstyle{definition}

\newaliascnt{remark}{equation}
\newtheorem{remarkthm}[remark]{Remark}
\newenvironment{remark}[1][]{\begin{remarkthm}[#1]\pushQED{\qed}}{\popQED \end{remarkthm}}
\aliascntresetthe{remark}

\newaliascnt{defn}{equation}
\newtheorem{defnthm}[defn]{Definition}
\newenvironment{defn}[1][]{\begin{defnthm}[#1]\pushQED{\qed}}{\popQED \end{defnthm}}
\aliascntresetthe{defn}

\newaliascnt{example}{equation}
\newtheorem{examplethm}[example]{Example}
\newenvironment{example}[1][]{\begin{examplethm}[#1]\pushQED{\qed}}{\popQED \end{examplethm}}
\aliascntresetthe{example}

\usepackage{cleveref}

\crefname{theorem}{Theorem}{Theorems}
\crefname{prop}{Proposition}{Propositions}
\crefname{section}{Section}{Sections}
\crefname{subsection}{Section}{Sections}
\crefname{defn}{Definition}{Definitions}
\crefname{remark}{Remark}{Remarks}
\crefname{lemma}{Lemma}{Lemmas}
\crefname{corollary}{Corollary}{Corollaries}
\crefname{example}{Example}{Examples}
\crefname{figure}{Figure}{Figures}

\newcommand{\mb}[1]{\mathbf{#1}}

\DeclareMathOperator{\Ad}{Ad}

\DeclareMathOperator{\Hom}{Hom}

\DeclareMathOperator{\End}{End}

\DeclareMathOperator{\Spec}{Spec}
\DeclareMathOperator{\Ker}{Ker}
\DeclareMathOperator{\tr}{tr}

\DeclareMathOperator{\coind}{coind}

\DeclareMathOperator{\Gr}{Gr}

\DeclareMathOperator{\Hol}{Hol}

\renewcommand{\Re}{\operatorname{Re}}

\newcommand{\wt}{\widetilde}
\newcommand{\wh}{\widehat}

\newcommand{\nin}{\notin}
\newcommand{\comments}[1]{}
\newcommand{\ol}[1]{\overline{#1}}

\newcommand{\al}{\alpha}
\newcommand{\be}{\beta}
\newcommand{\de}{\delta}
\newcommand{\De}{\Delta}

\newcommand{\ga}{\gamma}
\newcommand{\Ga}{\Gamma}
\newcommand{\io}{\iota}

\newcommand{\la}{\lambda}
\newcommand{\La}{\Lambda}
\newcommand{\om}{\omega}

\renewcommand{\th}{\theta}

\newcommand{\si}{\sigma}
\newcommand{\Si}{\Sigma}
\newcommand{\vphi}{\varphi}

\newcommand{\del}{\partial}

\newcommand{\CA}{{\mathcal A}}
\newcommand{\CB}{{\mathcal B}}
\newcommand{\CC}{{\mathcal C}}

\newcommand{\CL}{{\mathcal L}}
\newcommand{\CM}{{\mathcal M}}
\newcommand{\CN}{{\mathcal N}}

\newcommand{\CR}{{\mathcal R}}
\newcommand{\CS}{{\mathcal S}}

\newcommand{\CW}{{\mathcal W}}
\newcommand{\CX}{{\mathcal X}}

\newcommand{\BB}{{\mathbb B}}
\newcommand{\BC}{{\mathbb C}}

\newcommand{\BN}{{\mathbb N}}

\newcommand{\BP}{{\mathbb P}}

\newcommand{\BR}{{\mathbb R}}
\newcommand{\BS}{{\mathbb S}}

\newcommand{\BZ}{{\mathbb Z}}

\newcommand{\fb}{{\mathfrak b}}

\newcommand{\fs}{{\mathfrak s}}

\newcommand{\xqedhere}[2]{
  \rlap{\hbox to#1{\hfil\llap{\ensuremath{#2}}}}}

\newcommand\mono{\hookrightarrow}
\newcommand\into\mono
\newcommand\epi{\twoheadrightarrow}
\newcommand\onto\epi

\newcommand\<\langle
\renewcommand\>\rangle

\newcommand{\stdknet}{\CW_N}
\newcommand{\stdjnet}{\CW_K}
\newcommand{\bps}{\BS}
\newcommand{\qr}{\si}
\newcommand{\ti}{\tilde}
\newcommand{\nab}{\nabla}
\newcommand\ab{{\mathrm{ab}}}
\newcommand\fro{{\overline{\underline{\Omega}}}}
\newcommand\fweb{W}
\newcommand\basis{\BB}

\newcommand\wht{\mathrm{wt}}

\newcommand\dmod{\textrm{-mod}}
\newcommand\Jmod{J\textrm{-mod}}
\newcommand\Jopmod{J^{op}\textrm{-mod}}
\DeclarePairedDelimiter{\floor}{\lfloor}{\rfloor}

\newcommand\ba{\mathbf{a}}

\newcommand\bi{\mathbf{i}}
\newcommand\bX{\mathbf{X}}
\newcommand\bF{\mathbf{F}}
\newcommand\bD{\mathbf{D}}

\begin{document}
\title{Toda Systems, Cluster Characters, and Spectral Networks}
\author{Harold Williams}
\address{Harold Williams\newline
University of Texas at Austin\newline
Department of Mathematics\newline
Austin TX 78712\newline
USA}
\email{hwilliams@math.utexas.edu}

\begin{abstract}
We show that the Hamiltonians of the open relativistic Toda system are elements of the generic basis of a cluster algebra, and in particular are cluster characters of nonrigid representations of a quiver with potential.   Using cluster coordinates defined via spectral networks, we identify the phase space of this system with the wild character variety related to the periodic nonrelativistic Toda system by the wild nonabelian Hodge correspondence.  We show that this identification takes the relativistic Toda Hamiltonians to traces of holonomies around a simple closed curve.  In particular, this provides nontrivial examples of cluster coordinates on $SL_n$-character varieties for $n > 2$ where canonical functions associated to simple closed curves can be computed in terms of quivers with potential, extending known results in the $SL_2$ case. 
\end{abstract}

\maketitle

\section{Introduction}
The first purpose of this paper is to study a basic example of a cluster integrable system, the open relativistic Toda chain, from the point of view of the additive categorification of cluster algebras.  Following \cite{Goncharov2011}, by a cluster integrable system we mean an integrable system whose phase space is a cluster variety equipped with its canonical Poisson structure.  Examples include those studied in \cite{Goncharov2011,Fock2012,Hoffmann2000,Williams2013a}, and generally encompass those referred to as relativistic integrable systems in the literature \cite{Ruijsenaars1990}.  Roughly, cluster varieties are Poisson varieties whose coordinate rings (cluster algebras) are equipped with a canonical partial basis of functions called cluster variables \cite{Fomin2001}.  Additive categorification refers to the study of cluster algebras through the representation theory of associative algebras, in particular Jacobian algebras of quivers with potential \cite{Amiot2011,Keller2012}.  A key notion is that of the cluster character (or Caldero-Chapoton function) of a representation of a quiver with potential, a generating function of the Euler characteristics of its quiver Grassmannians \cite{Caldero2004,Caldero2006,Palu2008,Plamondon2011}.  In particular, while cluster variables are a priori defined in a recursive, combinatorial way, they can in retrospect be described in a nonrecursive, representation-theoretic way as the cluster characters of rigid indecomposable representations.  The notion of cluster character moreover provides a natural means of extending the set of cluster variables to a complete canonical basis of a cluster algebra, referred to as its generic basis, which includes cluster characters of nonrigid representations \cite{Dupont2011}. Our first main result is that the open relativistic Toda Hamiltonians, while not cluster variables, are nonetheless elements of the generic basis

We may realize the phase space of the open relativistic Toda chain as the quotient $SL_n^{c,c}/\Ad H$ of a Coxeter double Bruhat cell of $SL_n(\BC)$ by its Cartan subgroup $H$ \cite{Gekhtman2011}.  Double Bruhat cells are the left $H$-orbits of the symplectic leaves of the standard Poisson-Lie structure on a complex simple Lie group \cite{Hoffmann2000}, and are prototypical examples of cluster varieties \cite{Berenstein2005}.  The cluster structure on $SL_n^{c,c}/\Ad H$ is encoded by the quiver $Q_n$ shown in \cref{fig:Qn}.  The Toda Hamiltonians are the restrictions to $SL_n^{c,c}/\Ad H$ of the characters of the fundamental representations of $SL_n(\BC)$, and also appear as the conserved quantities of the $A_n$ $Q$-system \cite{Kedem2013,DiFrancesco2013}. The expansion of these Hamiltonians in cluster coordinates has a combinatorial description as a weighted sum of paths in a directed annular graph \cite{Gekhtman2011}, following a framework familiar from the theory of totally nonnegative matrices \cite{Fomin2000}.  To identify these Hamiltonians as cluster characters of representations of the relevant quiver with potential, we recall the coefficient quiver of a representation, which encodes the action of the path algebra on a chosen basis \cite{Ringel1998}.  By finding a suitable basis of the relevant nonrigid representations, we reduce the classification of subrepresentations to the enumeration of certain subquivers of the coefficient quiver, following an idea of \cite{Cerulli-Irelli2011}.  We then show that this problem is equivalent to the enumeration of paths in the directed graph used to describe the cluster structure on $SL_n^{c,c}/\Ad H$.  We note that being the cluster character of a representation is an extremely restrictive constraint on a Laurent polynomial, and indeed a generic cluster character is completely determined by its leading term.

\begin{figure}
\begin{tikzpicture}[thick,>=stealth']
\newcommand*{\Ddotsdist}{2}
\newcommand*{\shft}{1}
\newcommand*{\DrawDots}[1]{
  \fill ($(#1) + .25*(\Ddotsdist,0)$) circle (.03);
  \fill ($(#1) + .5*(\Ddotsdist,0)$) circle (.03);
  \fill ($(#1) + .75*(\Ddotsdist,0)$) circle (.03);
} 
\node [matrix] (Q) at (0,0)
{
\coordinate [label=below:2] (2) at (0,0);
\coordinate [label=below:4] (4) at (2,0);
\coordinate (6) at (4,0);
\coordinate (2n-4) at (6-\shft,0);
\coordinate [label=below:2n-2] (2n-2) at (8-\shft,0);
\coordinate [label=below:2n] (2n) at (10-\shft,0);

\coordinate [label=1] (1) at (0,2);
\coordinate [label=3] (3) at (2,2);
\coordinate (5) at (4,2);
\coordinate (2n-5) at (6-\shft,2);
\coordinate [label=2n-3] (2n-3) at (8-\shft,2);
\coordinate [label=2n-1] (2n-1) at (10-\shft,2);

\foreach \v in {1,2,3,4,2n-3,2n-2,2n-1,2n} {\fill (\v) circle (.06);};
\foreach \s/\t in {1/2,3/4,2n-3/2n-2,2n-1/2n} {
  \draw [->,shorten <=1.7mm,shorten >=1.7mm] ($(\s)+(0.06,0)$) to ($(\t)+(0.06,0)$);
  \draw [->,shorten <=1.7mm,shorten >=1.7mm] ($(\s)-(0.06,0)$) to ($(\t)-(0.06,0)$);
};
\foreach \s/\t in {2/3,4/5,2n-4/2n-3,2n-2/2n-1} {
  \draw [->,shorten <=1.7mm,shorten >=1.7mm] ($(\s)$) to ($(\t)$);
};
\foreach \s/\t in {4/1,6/3,2n-2/2n-5,2n/2n-3} {
  \draw [shorten <=1.7mm,shorten >=1.7mm] ($(\s)$) to ($(\s)!.5!(\t)$);
  \draw [->,shorten <=1.7mm,shorten >=1.7mm] ($(\s)!.5!(\t)$) to ($(\t)$);
};
\DrawDots{4-\shft*.5,1};\\
};
\node (equals) [left=0mm of Q] {$Q_n = $};
\node (equalsr) [right=0mm of Q] {\quad\quad\quad};
\end{tikzpicture}
\caption{The quiver $Q_n$, introduced in the context of $Q$-systems in \cite{Kedem2013} and in the context of gauge theory in \cite{Cecotti2010,Fiol2000}.}\label{fig:Qn}
\end{figure}
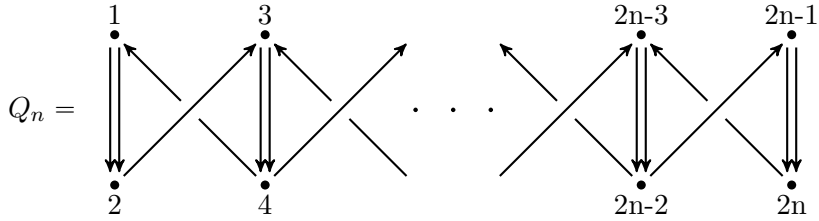

Our second main result uses cluster coordinates to identify the phase space $SL_n^{c,c}/\Ad H$ with a wild character variety \cite{Boalch2013} in such a way that the relativistic Toda Hamiltonians are identified with traces of holonomies around a simple closed curve.  The relevant wild character variety is related by the wild nonabelian Hodge correspondence to the periodic (nonrelativistic) Toda system, viewed as a meromorphic Hitchin system on $\BP^1$.  The cluster structure on this character variety is defined via spectral networks \cite{Gaiotto}, combinatorial objects built out of trajectories of differential equations defined by a Hitchin spectral curve.  It is well-known on physical grounds that the cluster structure on this particular wild character variety is encoded by the quiver $Q_n$ \cite{Cecotti2010}, identifying it birationally with $SL_n^{c,c}/\Ad H$ by identifying their respective cluster coordinates.  Proving that this isomorphism identifies the relativistic Toda Hamiltonians with traces of holonomies requires first extracting a sufficiently explicit description of the spectral networks defined implicitly by the periodic Toda spectral curves.  We then show that the combinatorics of weighted paths in directed graphs used to compute the relativistic Toda Hamiltonians is directly reproduced by the path-lifting scheme of \cite{Gaiotto}, with the relevant directed graph appearing as the 1-skeleton of the periodic Toda spectral curve.

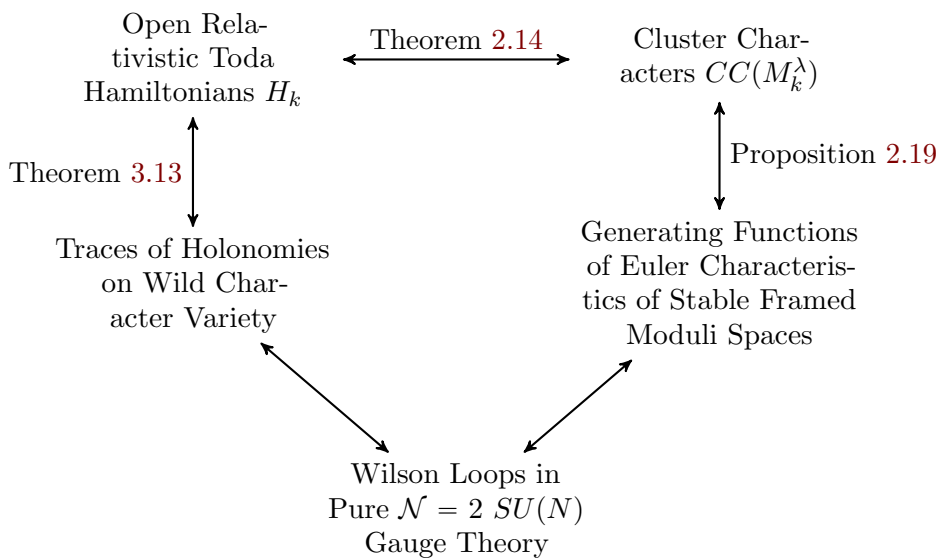
\begin{figure}
\begin{tikzpicture}
[thick,>=stealth',box/.style={text width=3.7cm, align=center}]
\newcommand*{\vs}{-3}
\newcommand*{\hs}{7}
\newcommand*{\tw}{3}
\node (a) at (0,0) [box] {Open Relativistic Toda Hamiltonians $H_k$};
\node (b) at (\hs,0) [box] {Cluster Characters $CC(M_k^\la)$};
\node (c) at (0,\vs) [box] {Traces of Holonomies on Wild Character Variety};
\node (d) at (\hs,\vs) [box] {Generating Functions of Euler Characteristics of Stable Framed Moduli Spaces};
\node (e) at (.5*\hs,2*\vs) [box] {Wilson Loops in Pure $\CN=2$ $SU(N)$ Gauge Theory};
\draw[<->] (a) -- (b) node[midway,above]{\cref{thm:heqcc}};
\draw[<->] (a) -- (c) node[midway,left]{\cref{thm:holotonian}};
\draw[<->] (b) -- (d) node[midway,right]{\cref{thm:framedpotential}};
\draw[<->] (c) -- (e);
\draw[<->] (d) -- (e);
\end{tikzpicture}
\caption{Our main results are \cref{thm:holotonian,thm:heqcc}, identifying the relativistic Toda Hamiltonians on $SL_n^{c,c}/\Ad H$ with traces of holonomies and cluster characters, respectively. Applying a standard correspondence in \cref{thm:framedpotential}, the resulting equality of the expressions in the middle row can be interpreted physically as the agreement of two different kinds of formula for the expectation values of supersymmetric line operators \cite{Gaiotto2010}.}
\label{fig:summary}
\end{figure}

Our reason for considering the above two results in conjunction is that they prove a natural expectation regarding canonical bases in coordinate rings of (possibly wild) $SL_n(\BC)$-character varieties of punctured surfaces: that the expansion in cluster coordinates of the trace in an $SL_n(\BC)$-representation of the holonomy around a simple closed curve is the cluster character of a representation of the associated quiver with potential.  While this is known in large generality for $SL_2(\BC)$-character varieties \cite{Haupt2012}, for $n >2$ this is the first example of such a result.  Indeed, the wild $SL_n(\BC)$-character varieties we consider are the simplest ones whose underlying surface has nontrivial fundamental group, hence provide a natural first step beyond the $SL_2$ case.  In particular, it follows from our results that in the example we consider the traces of holonomies in the fundamental $SL_n(\BC)$-representations are elements of the generic basis of \cite{Dupont2011,Cerulli-Irelli2012}.  We note that explicitly considering traces in general irreducible $SL_n(\BC)$-representations would add additional subtleties we do not consider here; the quiver-theoretic description of these functions requires corrections to the relevant generic cluster character, either in the form of a suitable nongeneric cluster character or an application of the Coulomb branch formula \cite{Cordova2013}.  On $SL_2(\BC)$-character varieties, traces of holonomies are also examples of Hamiltonians of cluster integrable systems.  Hamiltonians of more general cluster integrable systems such as in \cite{Goncharov2011} provide a distinct direction in which our result should generalize, orthogonal to that of traces of holonomies on other $SL_n(\BC)$-character varieties.

For context, we recall the gauge-theoretic interpretation of these ideas in terms of line operators in 4d $\CN=2$ theories of class $\CS$ \cite{Gaiotto2010,Cordova2013,Cirafici2013,Xie2013,Chuang2013,DelZotto}.  Theories of this type are constructed out of a Riemann surface $C$ with irregular data, the wild character variety of $C$ appearing as a space of vacua of the theory compactified on $S^1$.   Expectation values of supersymmetric line operators wrapping $S^1$ give rise to holomorphic functions on the character variety.  As argued in \cite{Gaiotto2010}, the expansions of these functions in cluster coordinates should be generating functions of framed BPS indices.  The quiver which encodes the cluster structure on the character variety is the BPS quiver of the theory \cite{Alim2011}, and framed BPS indices are roughly Euler characteristics of moduli spaces of stable framed representations of the BPS quiver (with corrections, in general) \cite{Cordova2013}.  A standard correspondence between quiver Grassmannians and moduli spaces of stable framed representations identifies these generating functions with cluster characters of unframed representations \cite{Nagao2013,Reineke}.  On the other hand, one construction of such line operators is through surface operators of the 6d $(2,0)$ theory partially wrapping curves in $C$.  In the case of a simple closed curve the resulting expectation value should be the trace in some representation of the holonomy around this curve \cite{Gaiotto2010}, hence this trace should be a cluster character.  Of course, the quiver $Q_n$ is familiar as the BPS quiver of pure $\CN=2$ $SU(N)$ Yang-Mills theory \cite{Cecotti2010,Alim2011,Cecotti2012}, hence the wild character variety related to $SL_n^{c,c}/\Ad H$ is the one associated with its Seiberg-Witten system, the periodic Toda system.  From this point of view, the open relativistic Toda Hamiltonians are the expectation values of Wilson loops in the fundamental representations of $SU(N)$ \cite{Gaiotto2010}.  While other works such as \cite{Cirafici2013,Cordova2013,DelZotto} consider the problem of computing framed BPS spectra in general cases, our goal in this paper is to prove the nontrivial agreement of geometric and representation-theoretic approaches in a case where it is possible to perform exact computations with the relevant spectral networks.

An important feature of our story is that the phase spaces of the two ``opposite'' Toda systems we consider, open relativistic and periodic nonrelativistic, are directly connected through the wild nonabelian Hodge correspondence.  This should be viewed in light of the following picture, which we will not attempt to make precise.  The phase space of the periodic relativistic Toda chain, another example of a cluster integrable system, can be identified with a space of monopoles on $S_x^1 \times S_y^1 \times \BR^{}_z$ \cite{Cherkis2012}.  More precisely, this space is hyperk\"ahler with complex structures $I$, $J$, and $K$ in correspondence with the factors $S^1_x$, $S^1_y$, and $\BR^{}_z$, and in both structures $I$ and $J$ it admits a spectral decomposition identifying it holomorphically with the periodic relativistic Toda system (of course, these are different identifications).  Letting the radius of $S^1_x$ shrink to $0$, the space degenerates to a space of solutions to Hitchin equations on $S_y^1 \times \BR_z$.  This limit is felt differently by the complex structures $I$ and $J$, breaking the symmetry between them.  In complex structure $I$ we have the nonrelativistic limit between the two flavors of periodic Toda chain.  In complex structure $J$ we have a degeneration between the two flavors of relativistic Toda chain, presumably coinciding with that observed in \cite{Fock2012}.  In particular, the periodic relativistic Toda phase space is itself a cluster integrable system, the cluster structure being encoded by a bipartite graph naturally identified with the 1-skeleton of its spectral curves (this is a general feature of the systems of \cite{Goncharov2011}, see also \cite{Feng2008}).  The above picture then provides a conceptual link between this fact and our identification of the bicolored graph encoding the cluster structure on the open relativistic Toda phase space with the 1-skeleton of the periodic nonrelativistic Toda spectral curves.  

The representation theory of the Jacobian algebra of $Q_n$ was studied in \cite{Cecotti2013}, where in particular the stable representations at weak coupling were classified and proved consistent with known perturbative features of the pure gauge theory BPS spectrum.  A distinguished role in its representation theory is played by the light subcategory $\CL$, a $\BP^1$-family of orthogonal subcategories $\CL_\la$ each equivalent to the module category of the $A_n$ preprojective algebra $\La_n$ \cite{Cecotti2013a}.  The objects of each $\CL_\la$ that are stable for some weakly-coupled stability condition are those corresponding to $\La_n$-modules whose dimension vector is a positive root.  On the other hand, the representations corresponding to the relativistic Toda Hamiltonians are recovered as cluster characters of the objects of $\CL_\la$ corresponding to projective-injective $\La_n$-modules.  This connection offers useful intuition for \cref{thm:heqcc}, which in this light links the projective-injective $\La_n$-modules to the characters of the fundamental representations of $SL_n(\BC)$.  Indeed, that these objects are connected is already well-known. We recall in particular the result of \cite{Baumann,Baumanna} that the convex hull of the dimension vectors of submodules of a projective-injective $\La_n$-module is a Weyl polytope, the convex hull of the weights of a fundamental representation of $SL_n(\BC)$.  

\textsc{Acknowledgments}  I thank Andy Neitzke, Michele Del Zotto, and Dylan Rupel for useful discussions.  This work was partially supported by NSF grants DMS-12011391 and DMS-1148490, and the Centre for Quantum Geometry of Moduli Spaces at Aarhus University.

\section{Relativistic Toda Systems and Quivers with Potential}\label{sec:reltodaandreps}

\subsection{Relativistic Toda Hamiltonians and Nonintersecting Paths}\label{sec:paths}

In this section we recall the construction of the open relativistic Toda system in terms of the Coxeter double Bruhat cell $SL_{n+1}^{c,c}/\Ad H$ \cite{Hoffmann2000,Williams2013c}.  This in particular identifies its phase space as the cluster variety associated with the quiver $Q_{n}$.  Following \cite{Gekhtman2011}, we identify the expansion of its Hamiltonians in cluster coordinates with generating functions of collections of nonintersecting paths on a directed graph in the annulus.  This could be reformulated in terms of perfect matchings, as in the treatment of torus graphs in \cite{Goncharov2011,Fock2012}, however directed paths will prove more natural for making contact with the path-lifting formalism of \cite{Gaiotto}.  However, we note in passing that when expressed in terms of perfect matchings \cref{thm:heqcc} can be recognized as a cousin of \cite[Theorem~5.6]{Mozgovoy2010}, though many details differ in each setting.  The cluster algebra notation we use is summarized in \cref{sec:clusterapp}.

Let
\[
c = s_1s_2\cdots s_{n}
\]
be the standard Coxeter element of $S_{n+1}$ and let $SL_{n+1}^{c,c}$ be the double Bruhat cell
\[
SL_{n+1}^{c,c} = B_+\dot{c}B_+ \cap B_-\dot{c}B_-,
\]
where $B_\pm$ are the subgroups of upper- and lower-triangular matrices and $\dot{c}$ any representative of $c$ in $SL_{n+1}$ (by $SL_{n+1}$ we will always mean $SL_{n+1}(\BC)$).  The quotient $SL_{n+1}^{c,c}/\Ad H$ of $SL_{n+1}^{c,c}$ under conjugation by the Cartan subgroup has dimension $2n$ and inherits a symplectic structure from the standard Poisson structure on $SL_{n+1}$.%\footnote{As in \cite{Williams2013c} we only care about well-behaved open toric subsets of $SL_{n+1}^{c,c}/\Ad H$, so do not worry about defining the quotient carefully.  To be conservative, we could simply take $SL_{n+1}^{c,c}/\Ad H$ to mean the quotient of the open subspace which is the union of cluster charts coming from double reduced words, as $H$ acts freely on this subspace.}  

\begin{defn}
The open relativistic Toda system is the integrable system with phase space $SL_{n+1}^{c,c}/\Ad H$ and Hamiltonians $H_1,\dotsc,H_{n}$ the restrictions of the characters of the fundamental representations $\bigwedge^k \BC^{n+1}$.
\end{defn}

In practice we will only consider $SL_{n+1}^{c,c}/\Ad H$ birationally.  There is nothing disinguished about our choice of Coxeter element beyond notational convenience, and other choices will yield birational but not necessarily biregular realizations of the phase space.  Similarly, since all our computations take place on open cluster charts in  $SL_{n+1}^{c,c}$ on which $H$ acts freely, we will not take care to define the quotient carefully.  
  
Each double Bruhat cell has a cluster structure with a subset of clusters indexed by double reduced words.  Recall that a double reduced word for a pair $(u,v)$ of Weyl group elements is a shuffle of reduced words for $u$ and $v$; we fix the double reduced word
\[
\mb{i} = (1,-1,2,-2,\dots,n,-n)
\]
for $(c,c)$.  More precisely, the double Bruhat cells $SL_{n+1}^{c,c}$ and $PSL_{n+1}^{c,c}$ of the simply-connected and adjoint groups have $\CA$- and $\CX$-type cluster structures.  The cluster $\CX$-variety structure on $PSL_{n+1}^{c,c}$ descends to one on $PSL_{n+1}^{c,c}/\Ad H$ by amalgamation \cite{Williams2013c}.  The double reduced word $\mb{i}$ encodes a seed of the cluster structure on $PSL_{n+1}^{c,c}/\Ad H$ whose associated quiver is $Q_n$ from \cref{fig:Qn}.  The toric cluster chart associated with this seed is given by the map $\CX_{Q_{n}} \to PSL_{n+1}^{c,c}/\Ad H$ defined by 
\[
(y_1,\dots,y_{2n}) \mapsto E_1 y_1^{\om_1^\vee} F_1 y_2^{\om_1^\vee} \cdots E_n y_{2n-1}^{\om_n^\vee} F_{2n} y_{2n}^{\om_{2n-1}^\vee}.
\]
Here the right-hand side is an ordered product of simple root and fundamental coweight subgroups.
That is, 
\begin{gather*}
(E_i)_{jk} = \begin{cases} 1 & \text{$j=k$ or $(j,k)=(i,i+1)$}\\ 0 & \text{otherwise} \end{cases}, \quad
(F_i)_{jk} = \begin{cases} 1 & \text{$j=k$ or $(j,k)=(i+1,i)$}\\ 0 & \text{otherwise} \end{cases}\\
(y^{\om_i^\vee})_{jk} = \begin{cases} y^{\frac{n+1-k}{n+1}} & \text{$j=k$ and $j \leq i$}\\y^{\frac{-k}{n+1}} & \text{$j=k$ and $j > i$}\\ 0 & \text{otherwise} \end{cases}.
\end{gather*}
It is often convenient to factor $y^{\om_k^\vee}$ as the product of a scalar matrix and a diagonal matrix whose entries are equal to either 1 or $y$:  
\[
y^{\om_k^\vee} = y^{\frac{-k}{n+1}} \begin{pmatrix}
y & 0 & & \cdots & & 0\\
0 & \ddots &&&&\\
 & & y &&&\\
\vdots & & & 1 & & \vdots\\
 & & & & \ddots & 0\\
0 & & & \cdots& 0& 1
\end{pmatrix}.
\] 

We want to lift these to coordinates on $SL_{n+1}^{c,c}/\Ad H$.  Define a new torus $T$ with formal coordinates $y_1^{\frac{\pm 1}{n+1}},\cdots,y_{2n}^{\frac{\pm 1}{n+1}}$.  The inverse of the adjacency matrix of $Q_n$ has entries that are rational with denominator $n+1$, so the obvious map $T \to \CX_{Q_n}$ factors through the canonical map $p_{Q_n}: \CA_{Q_n} \to \CX_{Q_n}$.  Now define $T \xrightarrow{} SL_{n+1}^{c,c}/\Ad H$ by
\[
(y^{\frac{1}{n+1}}_1,\dots,y^{\frac{1}{n+1}}_{2n}) \mapsto E_1 (y^{\frac{1}{n+1}}_1)^{(n+1)\om_1^\vee} \cdots F_{2n} (y^{\frac{1}{n+1}}_{2n})^{(n+1)\om_{2n-1}^\vee}.
\]
It satisfies
\[
\begin{tikzcd}
T \arrow{r}{} \arrow{d}{} & SL_{n+1}^{c,c}/\Ad H \arrow{d}{} \\
\CX_{Q_n} \arrow{r}{} & PSL_{n+1}^{c,c}/\Ad H,
\end{tikzcd}
\] 
where the right-hand map is the quotient map, and factors through a map $\CA_{Q_n} \to SL_{n+1}^{c,c}/\Ad H$ (though neither it nor the latter are injective).  This factorization can be deduced from the following diagram, where $\wt{Q}_n$ is the quiver of the cluster chart on $SL_{n+1}^{c,c}$ associated to $\mb{i}$ and the commutativity of the leftmost square is a universal feature of amalgamation:
\[
\begin{tikzcd}
\CA_{Q_n} \arrow[hookrightarrow]{r}{} \arrow[two heads]{d}{p_{Q_n}} & \CA_{\wt{Q}_n} \arrow[two heads]{d}{p_{Q_n}} \arrow{r}{} & SL^{c,c}_{n+1} \arrow[two heads]{d}{\wt{\pi}} \arrow[two heads]{r}{} & SL^{c,c}_{n+1}/\Ad H \arrow[two heads]{d}{\wt{\pi}} \\
\CX_{Q_n} & \CX_{\wt{Q}_n} \arrow[two heads]{l}{} \arrow{r}{} & PSL^{c,c}_{n+1} \arrow[two heads]{r}{} & PSL^{c,c}_{n+1}/\Ad H  
\end{tikzcd}
\]
The maps $\wt{\pi}$ are the composition of the quotient map and the twist automorphism $\tau$ of $SL_{n+1}^{c,c}$, but note that by the results of \cite{Williams2013c} we have $\tau^*H_i = H_i$. 

%\[
%\begin{tikzcd}
%\CA_{Q_n} \arrow{r}{a_{\bi}} \arrow{d}{p_{Q_n}} & SL_{n+1}^{c,c}/\Ad H \arrow{d}{} \\
%\CX_{Q_n} \arrow{r}{} & PSL_{n+1}^{c,c}/\Ad H
%\end{tikzcd}
%\] 

%As the adjacency matrix of $Q_n$ is nondegenerate, the canonical map $\CA_{Q_n} \to \CX_{Q_n}$ is a finite cover.  We identify $\CA_{Q_n}$ as an open chart on $SL_{n+1}^{c,c}/\Ad H$ so that its map to $\CX_{Q_n}$ is compatible with the finite cover $SL_{n+1}^{c,c}/\Ad H \to PSL_{n+1}^{c,c}/\Ad H$.  This will define for us the Laurent expansions of the Hamiltonians $H_i$ on $\CA_{Q_n}$, which will be our principle object of study.  We let $y_i^{\frac{1}{n+1}} \in \BC[\CA_{Q_n}]$ be the unique root of $p^*_{Q_n} y_i$ which is a monomial in $x_1,\dotsc,x_n$ (that is, with coefficient 1).  Then we define a map $a_{\bi}:\CA_{Q_n} \xrightarrow{} SL_{n+1}^{c,c}/\Ad H$ by
%\[
%(y^{\frac{1}{n+1}}_1,\dots,y^{\frac{1}{n+1}}_{2n}) \mapsto E_1 (y^{\frac{1}{n+1}}_1)^{(n+1)\om_1^\vee} \cdots F_{2n} (y^{\frac{1}{n+1}}_{2n})^{(n+1)\om_{2n-1}^\vee},
%\]
%which satisfies
%\[
%\begin{tikzcd}
%\CA_{Q_n} \arrow{r}{a_{\bi}} \arrow{d}{p_{Q_n}} & SL_{n+1}^{c,c}/\Ad H \arrow{d}{} \\
%\CX_{Q_n} \arrow{r}{} & PSL_{n+1}^{c,c}/\Ad H
%\end{tikzcd}
%\] 

We can describe the cluster coordinates on $SL_{n+1}^{c,c}/\Ad H$ combinatorially following a standard construction in the theory of total nonnegativity.  We define directed graph $\CN_{\mb{i}}$ via the following embedding into $[0,1]^2$:
\begin{align*}
\CN_{\mb{i}} =& \left( \bigcup_{k=1}^n [0,1]\times\{\frac{k}{n+2}\} \right) \cup \left( \bigcup_{k=1}^n \{\frac{k}{2n+1}\}\times [n+1-k,n+2-k] \right)\\&\quad \cup \left( \bigcup_{k=1}^n \{\frac{n+k}{2n+1}\}\times [n+1-k,n+2-k] \right).
\end{align*}
The vertical edges of the form $\{\frac{k}{2n+1}\}\times[n+1-k,n+2-k]$ are directed upward, those of the form $\{\frac{n+k}{2n+1}\}\times[n+1-k,n+2-k]$ are directed downward, and the horizontal edges are directed leftward.\footnote{Readers familiar with this formalism will note that the graph $\CN_{\mb{i}}$ as described corresponds to the double reduced word $\mb{i}' = (1,\dots,r,-1,\dots,-r)$ rather than $\mb{i} = (1,-1,2,-2,\dots,r,-r)$.  However, these double reduced words are equivalent up to a trivial reindexing, and their associated graphs are equivalent in the sense described.  The graph associated with $\mb{i}'$ is more convenient to draw, while the graph associated with $\mb{i}$ induces a more convenient indexing of the vertices of $Q_n$.}  See and \cref{ex:Q1,ex:Q2,ex:H3} for illustrations.  We refer to $(1,\frac{k}{n+2})$, $(0,\frac{k}{n+2})$ as the $k$th input and output vertices of $\CN_{\mb{i}}$, respectively.  We only need to consider the graph $\CN_{\mb{i}}$ up to isotopy and the following move.  We allow two vertices with one incoming edge and two outgoing edges (or vice-versa) to collide and expand in the opposite configuration as pictured:
\[
\begin{tikzpicture}[thick,>=stealth',decoration={snake,amplitude=1}]
\node [matrix] (left) at (0,0) {
\draw [->] (2,0) -- (0,0);
\draw [->] (.6,0) -- (.6,.6);
\draw [<-] (1.4,-.6) -- (1.4,0);\\};

\node [matrix] (right) at (5,0) {
\draw [->] (2,0) -- (0,0);
\draw [->] (1.4,0) -- (1.4,.6);
\draw [<-] (.6,-.6) -- (.6,0);\\};

\draw [<->,decorate] ($(left.east)+(.3,0)$) -- ($(right.west)+(-.3,0)$);
\end{tikzpicture}
\]
We label the component of $[0,1]^2 \smallsetminus \CN_{\mb{i}}$ to the immediate right of the edges $\{\frac{k}{2n+1}\}\times[n+1-k,n+2-k]$, $\{\frac{n+k}{2n+1}\}\times [n+1-k,n+2-k]$ by the variables $y_k$, $y_{n+k}$, respectively.

The graph $\CN_{\mb{i}}$ we describes a lift of the map $T \xrightarrow{} SL_{N}^{c,c}/\Ad H$ to a map  $T \xrightarrow{} SL_{n+1}^{c,c}$ in the following way.  By a directed path in $\CN_{\mb{i}}$ we mean a path whose orientation agrees with that of $\CN_{\mb{i}}$.  A maximal direct path $p$ starts and ends at boundary vertices of $\CN_{\mb{i}}$, and we define its weight $\wht(p)$ by the following conditions.  The weight of the lowest directed path $p_{\min}$ (the unique directed path from the $n$th input to the $n$th output) is
\begin{align*}
\wht(p_{\min})&=\prod_{i=1}^n(y_{2i-1}y_{2i})^{-\om_n(\om_i^\vee)}\\
&=\prod_{i=1}^n(y_{2i-1}y_{2i})^{-\frac{i}{n+1}}
\end{align*}
For any other directed path $p$ the ratio $\wht(p)/\wht(p_{\min})$ is the product of all $y_i$ that label components of $[0,1]^2 \smallsetminus \CN_{\mb{i}}$ lying between $p$ and $p_{\min}$ in $[0,1]^2$.  The $ij$th matrix entry of a point in the image of $T \xrightarrow{} SL_{n+1}^{c,c}$ is then the weighted sum
\[
\sum_{p:j \to i} \wht(p)
\]
over all directed paths from the $j$th input to the $i$th output of $\CN_{\mb{i}}$.  

Consider the map $[0,1]^2 \onto S^1 \times [0,1]$ given by identifying the left and right edges of $[0,1]^2$, and let $\ol{\CN}_{\mb{i}} \subset S^1 \times [0,1]$ be the image of $\CN_\bi$.  The preimage of any closed directed path in $\ol{\CN}_{\mb{i}}$ is a directed path in $\CN_{\mb{i}}$ from the $k$th input to the $k$th output for some $1\leq k \leq n$.  We define the weight of a closed directed path in $\ol{\CN}_\bi$ to be the weight of its preimage.  Note that each $y_i$ now labels a contractible component of $(S^1 \times [0,1]) \smallsetminus \ol{\CN}_{\mb{i}}$.  The counterclockwise-oriented boundaries of these components define a basis of 
\[
\Ker(H_1(\ol{\CN}_{\mb{i}},\BZ) \to H_1(S^1 \times [0,1],\BZ)),
\]
and for two closed directed paths $p$, $p'$, the ratio $\wht(p)/\wht(p')$ is the class of $[p - p']$ written multiplicatively.  
If $P=\{p_\ell\}_{\ell = 0}^{i-1}$ is a nonintersecting $i$-tuple of closed directed paths in $\ol{\CN}_{\mb{i}}$, we define the weight of $P$ by
\[
\wht(P) = \prod_{\ell=0}^{i-1} \wht(p_\ell).
\]
The following proposition follows easily from Gessel-Viennot and the definition of the weight of a path.

\begin{prop}\label{prop:weightedsum}
In cluster coordinates on $\CA_{Q_n}$, the open relativistic Toda Hamiltonian $H_k$ is the weighted sum 
\[
\sum_{P = \{p_\ell\}_{\ell = 0}^{i-1}} \wht(P)
\]
of all nonintersecting $k$-tuples of closed directed paths in $\ol{\CN}_{\mb{i}}$.
\end{prop}

\begin{example}\label{ex:Q1}  For $SL_2^{c,c}/\Ad H$, the graph $\CN_{\mb{i}}$ is 
\[
\begin{tikzpicture}[thick,>=stealth',baseline={([yshift=-3pt]current bounding box.center)},anchor=base]
\newcommand*{\Xh}{0.45} % height of y_k in a face
\newcommand*{\ra}{0} % row a height
\newcommand*{\rb}{1} % row b height
\newcommand*{\vla}{.5} % line a position
\newcommand*{\vlb}{1.5} % line b position
\newcommand*{\length}{2.45} % length of network
\newcommand*{\varrowpos}{.54} % position of arrows on vertical lines
\node [matrix] (leftpic) at (0,0)
{
\foreach \y in {\ra,\rb} \draw [<-<] (0,\y) -- (\length,\y);

\coordinate (vlat) at (\vla,\rb);
\coordinate (vlas) at (\vla,\ra);
\coordinate (vlbt) at (\vlb,\ra);
\coordinate (vlbs) at (\vlb,\rb);

\draw (vlas) -- (vlat);
\draw [->] (vlas) -- ($(vlas)!\varrowpos!(vlat)$);
\draw (vlbs) -- (vlbt);
\draw [->] (vlbs) -- ($(vlbs)!\varrowpos!(vlbt)$);

\node at (1,\Xh) {$y_{1}$};
\node at (2,\Xh) {$y_{2}$};\\
};
\end{tikzpicture}
\]
and the cluster coordinates are given by
\[
(y_1,y_2) \mapsto y_1^{-\frac12}y_2^{-\frac12} \begin{pmatrix} y_2 + y_1y_2 & 1 \\ y_2 & 1 \end{pmatrix}.
\]

Computing the Hamiltonian $H_1$ requires taking the trace of the matrix on the right, which is the weighted sum of the three distinct closed directed paths in $\ol{\CN}_{\mb{i}}$.  Thus
\begin{align*}
H_1 &= y_1^{-\frac12}y_2^{-\frac12}(1 + y_2 + y_1y_2)\\
& = x_1x_2^{-1}(1 + y_2 + y_1y_2),
\end{align*}
where $y_1 = x_2^2$, $y_2 = x_1^{-2}$.
%\end{figure}
\end{example}

\begin{example}\label{ex:Q2}
For $SL_3^{c,c}/\Ad H$, the graph $\CN_{\mb{i}}$ is 
\[
\begin{tikzpicture}[thick,>=stealth',baseline={([yshift=-3pt]current bounding box.center)},anchor=base]
\newcommand*{\Xh}{0.45}
\newcommand*{\ra}{0} % row a height
\newcommand*{\rb}{1} % row b height
\newcommand*{\rc}{2} % row b height
\newcommand*{\length}{2.95} % length of network
\newcommand*{\vla}{.5} % line a position
\newcommand*{\vlb}{1} % line b position
\newcommand*{\vlc}{1.5} % line c position
\newcommand*{\vld}{2} % line d position
\newcommand*{\varrowpos}{.54} % position of arrows on vertical lines
\node [matrix] (leftpic) at (0,0) {

\coordinate (vlat) at (\vla,\rc);
\coordinate (vlas) at (\vla,\rb);
\coordinate (vlbt) at (\vlb,\rb);
\coordinate (vlbs) at (\vlb,\ra);
\coordinate (vlct) at (\vlc,\rb);
\coordinate (vlcs) at (\vlc,\rc);
\coordinate (vldt) at (\vld,\ra);
\coordinate (vlds) at (\vld,\rb);

\draw (vlas) -- (vlat);
\draw [->] (vlas) -- ($(vlas)!\varrowpos!(vlat)$);
\draw (vlbs) -- (vlbt);
\draw [->] (vlbs) -- ($(vlbs)!\varrowpos!(vlbt)$);
\draw (vlcs) -- (vlct);
\draw [->] (vlcs) -- ($(vlcs)!\varrowpos!(vlct)$);
\draw (vlds) -- (vldt);
\draw [->] (vlds) -- ($(vlds)!\varrowpos!(vldt)$);

\foreach \y in {\ra,\rb,\rc} \draw [<-<] (0,\y) -- (\length,\y);

\node at (1,\rb+\Xh) {$y_{1}$};
\node at (2,\rb+\Xh) {$y_{2}$};
\node at (1.5,\ra+\Xh) {$y_{3}$};
\node at (2.5,\ra+\Xh) {$y_{4}$};\\
};
\end{tikzpicture}
\]
and the cluster coordinates are given by
\[
(y_1,y_2,y_3,y_4) \mapsto y_1^{-\frac13}y_2^{-\frac13}y_3^{-\frac23}y_4^{-\frac23} \begin{pmatrix} y_2y_3y_4 + y_1y_2y_3y_4 & y_4 + y_3y_4 & 1 \\ y_2y_3y_4 & y_4 + y_3y_4 & 1 \\ 0 & y_4 & 1 \end{pmatrix}.
\]

There are two Hamiltonians $H_1$ and $H_2$ corresponding to the fundamental and anti-fundamental representations, respectively.  The former is a weighted sum of the five closed directed paths in $\ol{\CN}_{\mb{i}}$, while the latter is a weighted sum of the five nonintersecting pairs of closed directed paths:
\begin{align*}
H_1 &= y_1^{-\frac13}y_2^{-\frac13}y_3^{-\frac23}y_4^{-\frac23}(1 + y_4 + y_3y_4 + y_2y_3y_4 + y_1y_2y_3y_4) \\
&= x_3x_4^{-1}(1 + y_4 + y_3y_4 + y_2y_3y_4 + y_1y_2y_3y_4)\\
H_2 &= y_1^{-\frac23}y_2^{-\frac23}y_3^{-\frac13}y_4^{-\frac13}(1 + y_2 + y_1y_2 + y_1y_2y_4 + y_1y_2y_3y_4)\\
&= x_1x_2^{-1}(1 + y_2 + y_1y_2 + y_1y_2y_4 + y_1y_2y_3y_4).
\end{align*}
Here $y_1 = x_2^2x_4^{-1}$, $y_2 = x_1^{-2}x_3$, $y_3 = x_2^{-1}x_4^2$, and $y_4 = x_1x_3^{-2}$.
\end{example}

\subsection{The Jacobian Algebra of $Q_{n}$}\label{sec:qreps}

In this section we consider in detail the Jacobian algebra of $Q_n$ and construct certain representations $M_i^{\la}$ whose cluster characters recover the relativistic Toda Hamiltonians $H_i$.  We describe bases of the algebra $J(Q_n,W_n)$ and the modules $M_i^{\la}$ which will allow us to explicitly enumerate the submodules of $M_i^{\la}$, following the strategy of \cite{Cerulli-Irelli2011}.

We label the edges of $Q_n$ as follows.  For $i \in \{1,\dots,n\}$ the two vertical arrows from $2i-1$ to $2i$ are labeled $a_i$ and $b_i$.  For $i \in \{2,\dots,n\}$ the leftward diagonal arrows from $2i$ to $2i-3$ are labeled $\ell_i$, and for $i \in \{1,\dots,n-1\}$  the rightward diagonal arrows from $2i$ to $2i+1$ are labeled $r_i$.  

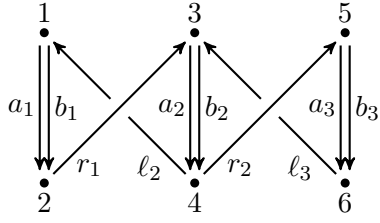
\begin{figure}
\begin{tikzpicture}[thick,>=stealth']
\newcommand*{\Ddotsdist}{2}
\newcommand*{\DrawDots}[1]{
  \fill ($(#1) + .25*(\Ddotsdist,0)$) circle (.03);
  \fill ($(#1) + .5*(\Ddotsdist,0)$) circle (.03);
  \fill ($(#1) + .75*(\Ddotsdist,0)$) circle (.03);
} 
\node [matrix] (Q) at (0,0)
{
\coordinate [label=below:2] (2) at (0,0);
\coordinate [label=below:4] (4) at (2,0);
\coordinate [label=below:6] (6) at (4,0);

\coordinate [label=1] (1) at (0,2);
\coordinate [label=3] (3) at (2,2);
\coordinate [label=5] (5) at (4,2);

\newcommand*{\hoff}{.3}
\node at (-\hoff,1) {$a_1$};
\node at (2-\hoff,1) {$a_2$};
\node at (4-\hoff,1) {$a_3$};
\node at (\hoff,1) {$b_1$};
\node at (2+\hoff,1) {$b_2$};
\node at (4+\hoff,1) {$b_3$};
\node at (.6,.2) {$r_1$};
\node at (2.6,.2) {$r_2$};
\node at (1.4,.2) {$\ell_2$};
\node at (3.4,.2) {$\ell_3$};

\foreach \v in {1,2,3,4,5,6} {\fill (\v) circle (.06);};
\foreach \s/\t in {1/2,3/4,5/6} {
  \draw [->,shorten <=1.7mm,shorten >=1.7mm] ($(\s)+(0.06,0)$) to ($(\t)+(0.06,0)$);
  \draw [->,shorten <=1.7mm,shorten >=1.7mm] ($(\s)-(0.06,0)$) to ($(\t)-(0.06,0)$);
};
\foreach \s/\t in {2/3,4/5} {
  \draw [->,shorten <=1.7mm,shorten >=1.7mm] ($(\s)$) to ($(\t)$);
};
\foreach \s/\t in {4/1,6/3} {
  \draw [shorten <=1.7mm,shorten >=1.7mm] ($(\s)$) to ($(\s)!.5!(\t)$);
  \draw [->,shorten <=1.7mm,shorten >=1.7mm] ($(\s)!.5!(\t)$) to ($(\t)$);
};\\
};
\end{tikzpicture}
\caption{The quiver $Q_3$ with labeled edges.}
\label{fig:quiveredgelables}
\end{figure}

We fix the potential
\[
W_n = \sum_{i=1}^{n-1} a_i\ell_{i+1}b_{i+1}r_i - b_i\ell_{i+1}a_{i+1}r_i.
\]
The cyclic derivatives of $W_n$ are
\begin{equation}\label{eq:cyclicderivatives}
\begin{aligned}
\del_{a_i} W_n = \ell_{i+1}b_{i+1}r_i - r_{i-1}b_{i-1}\ell_i \hspace{10mm}
\del_{b_i} W_n = r_{i-1}a_{i-1}\ell_i - \ell_{i+1}a_{i+1}r_i\\
\del_{\ell_i} W_n = b_ir_{i-1}a_{i-1} - a_ir_{i-1}b_{i-1} \hspace{10mm}
\del_{r_i} W_n = a_i\ell_{i+1}b_{i+1} - b_i\ell_{i+1}a_{i+1}.
\end{aligned}
\end{equation}
Here any terms referring to nonexistent edges such as $r_n$ or $a_0$ are understood to be zero.

As all relations are either paths or differences of paths, $J(Q_n,W_n)$ has a basis indexed by equivalence classes of paths in $Q_n$.  Generally, suppose $I$ is an ideal of a path algebra $\BC Q$ generated by relations of this form, that is
\[
I = \<p_1 - p'_1,\dots,p_k-p'_k,q_1,\dots,q_\ell\>
\]
where each pair $p_i$, $p'_i$ of paths has the same source and target.  Then the nonzero elements of
\[
\{ p + I: p\text{ a path in }Q\} \subset \BC Q/I
\]
form a basis of $\BC Q/I$.  Elements of this basis are labeled by equivalence classes of paths in $Q$ under the relation
\[
\al \sim \be \iff \textrm{$\al = ap_ib$, $\be = ap'_ib$ for some paths $a$, $b$ and some index $i$}. 
\]
The classes realized as labels of basis elements of $\BC Q/I$ are those with no representatives containing some $q_j$ as a subpath.  We call this the path basis of $\BC Q/I$.

To a path $p$ in $Q_n$ we associate a tuple $(s_p,t_p,\la_p,\rho_p,a_p,b_p)$, where $s_p$, $t_p \in Q_0$ are the source and target of $p$ and $\la_p$, $\rho_p$, $a_p$, $b_p$ the number of times $p$ traverses an edge labeled $\ell_k$, $r_k$, $a_k$, or $b_k$, respectively.  

\begin{prop}\label{prop:pathbasis} Elements of the path basis of $J(Q_n,W_n)$ are in bijection with the set of tuples such that
\begin{gather*}
\rho_p - \la_p = \left\lfloor\frac{t_p + 1}{2}\right\rfloor - \left\lfloor\frac{s_p + 1}{2}\right\rfloor\\ 
\rho_p \leq n-\left\lfloor\frac{s_p + 1}{2}\right\rfloor\\
\la_p < \left\lfloor\frac{s_p + 1}{2}\right\rfloor,
\end{gather*}
and such that $a_p+b_p$ equals the number of times any path with the given values of $s_p,t_p,\la_p,\rho_p$ traverses a vertical edge.
\end{prop}

\begin{proof}
Follows easily from the relations \labelcref{eq:cyclicderivatives}.  Informally, the relations say that if two paths differ only in the order in which they traverse $a_k$ and $b_k$ edges, but traverse the same number of each kind in total, then they are equivalent (and likewise for $\ell_k$, $r_k$ edges).  Moreover, the only paths that become zero in $J(Q_n,W_n)$ are those equivalent to a path that ``falls off the edge'' of $Q_n$.  
\end{proof}

The projective representation $P_i$ is the subspace of $J(Q_n,W_n)$ spanned by equivalence classes of paths with source $i$.  A path starting at $i$ and ending at $j$ is an element of the subspace $(P_i)_j$ supported at $j$.  

\begin{defn}
For  $i \in \{1,\dots,n\}$, define a $\BP^1$-family of modules $M_i^\la$ as follows.  Given projective coordinates $\la = (\la_1:\la_2)$, let $P_{2i} \xrightarrow{\la} P_{2i-1}$ be the map which sends the generator of $P_{2i}$ to the element $\la_1 a_i + \la_2 b_i \in (P_{2i-1})_{2i}$.  Then  $M_i^\la$ is the cokernel of $\la$:
\[
0 \to P_{2i} \xrightarrow{\la} P_{2i-1} \to M_i^\la \to 0.
\]
\end{defn}

\begin{remark}
We recall that the light subcategory $\CL$ of $J(Q_n,W_n)\dmod$ is a $\BP^1$-family of orthogonal subcategories $\CL_\la$ each equivalent to the module category of the $A_n$ preprojective algebra $\La_n$ \cite{Cecotti2013}.  The $M_i^\la$ are the modules corresponding to the projective-injective $\La_n$-modules under these equivalences.
\end{remark}

\begin{defn}\label{prop:Mbasis}
Let
\[
B_i^\la = \{ t_\la : t \in Q_0, 0 \leq \la \leq i-1, 2(i-\la)-1 \leq t \leq 2(n - \la)\}
\]
be the following basis of $M_i^\la$.  If $\la_2 \neq 0$, $t_\la$ is the image of the element of the path basis of $P_{2i-1}$ whose associated tuple $(2i-1,t_p,\la_p,\rho_p,a_p,b_p)$ has $t_p = t$, $\la_p=\la$, and $a_p=0$.  If $\la_2 =0$, we replace the condition $a_p=0$ with $b_p=0$.  
\end{defn}

Given \cref{prop:pathbasis} it is a straightforward computation to see that $B_i^\la$ is indeed a basis.

We recall from \cite{Ringel1998} the notion of a coefficient quiver.  Let $M$ be a representation of a quiver $Q$ and $B = \{b_i\}$ a basis of $M$ such that $B_v = B \cap M_v$ is a basis of $M_v$ for all $v \in Q_0$.  The coefficient quiver $ \Ga_B$ of $B$ has $(\Ga_B)_0 = B$ and an edge from $b_i \in B_v$ to $b_j \in B_w$ for every edge $a: v \to w$ of $Q$ such that $c_j \neq 0$ in the expansion
\[ b_i a = \sum_{b_k \in B_w} c_k b_k. \]

Let $\Ga_i^\la$ denote the coefficient quiver of the basis $B_i^\la$ of $M_i^\la$.  We say a subquiver $\Ga' \subset \Ga_i^\la$ is successor-closed if whenever $i$ is a vertex of $\Ga'$ and $a: i \to j$ an edge of $\Ga_i^\la$, $a$ and $j$ are also an edge and vertex of $\Ga'$.  The following proposition says that the submodule structure of $M_i^\la$ is completely determined by the coefficient quiver of $\Ga_i^\la$.  

\begin{prop}
Submodules of $M_i^\la$ are in bijection with successor-closed subquivers of $\Ga_i^\la$.  The submodule $N_{\Ga'}$ corresponding to a successor-closed subquiver $\Ga'$ is the subspace spanned by its vertices.
\end{prop}

\begin{proof}
The proposition is equivalent to the claim that for any submodule $N$ of $M_i^\la$, $N \cap B_i^\la$ is a basis of $N$.  For any $t \in (Q_0)_n$, define $E_t \in \End (M_i^\la)_t$ by
\[
E_t = \begin{cases} r_ja_{j+1}\ell_{j+1}a_{j} & t = 2j, \la_2 \neq 0 \\
r_jb_{j+1}\ell_{j+1}b_{j} & t = 2j, \la_2 = 0 \\
a_{j}r_ja_{j+1}\ell_{j+1} & t = 2j-1, \la_2 \neq 0 \\
b_{j}r_jb_{j+1}\ell_{j+1} & t = 2j-1, \la_2 = 0. \end{cases}
\]
Here again any terms referring to nonexistent edges such as $r_n$ are understood to be zero.  The basis $(B_i^\la)_t$ of $(M_i^\la)_t$ is of the form
\[
(B_i^\la)_t = \{t_k,t_{k+1},\dotsc,t_{m}\}
\]
for some $k$ and $m$.  It is clear that $E_t$ takes $t_\ell$ to $t_{\ell+1}$ for $\ell < m$ and takes $t_m$ to $0$.  It follows that $(B_i^\la)_t$ restricts to a basis of any subspace of $(M_i^\la)_t$ that is invariant under $E_t$, which in particular includes the subspace $N_t$ of any submodule $N$.
\end{proof}

In particular, all quiver Grassmannians of $M_i^\la$ are points; the corollary that the Euler characteristics of all its quiver Grassmannians are equal to 1 is Theorem 1 of \cite{Cerulli-Irelli2011} applied to the case at hand.

\begin{remark}
Note that if we replace all double edges of $\Ga_i^\la$ with single edges, we obtain the Hasse diagram of a partially-ordered set.  The proposition says that the ideal lattice of this poset is isomorphic with the submodule lattice of $M_i^\la$.  
\end{remark}

In anticipation of computing the cluster characters of the $M_i^\la$ let us compute now their injective resolutions.  Let $\si_n$ be the permutation of $(Q_n)_0$ defined by
\begin{gather*}
\si_n(2k) = 2k-1,\quad \si_n(2k-1) = 2k.
\end{gather*}
The potential $W_n$ on $Q_n$ induces an opposite potential $W_n^{op}$ on $Q_n^{op}$.  The permutation $\si_n$ induces an isomorphism of $Q$ and $Q^{op}$ which descends to an isomorphism of $J(Q_n,W_n)$ and $J(Q_n^{op},W_n^{op}) \cong J(Q_n,W_n)^{op}$.  We denote the corresponding equivalence of $\Jmod$ and $\Jopmod$ by $\si_n$ as well.  It is clear  that we have isomorphisms
\[
I_{\si_n(k)} \cong D\si_n P_k,
\]
where $D = \Hom_\BC(-,\BC)$.  We also let \[\nu_n: i \mapsto n+i-1\] denote the Nakayama involution of $\{1,\dotsc,n\}$. 

\begin{prop}\label{prop:injectiveresolution}
Given $1\leq j \leq n$ and $\la \in \BP^1$, let $I_{2j} \xrightarrow{\la} I_{2j-1}$ be the map obtained by applying $D\si_n$ to $P_{2j} \xrightarrow{\la} P_{2j-1}$.  Then we have an exact sequence
\[
0 \to M_{\nu_n(j)}^{\la} \to I_{2j} \xrightarrow{\la} I_{2j-1} \to 0.
\]
In particular, $M_{\nu_n(j)}^{\la}$ is the Auslander-Reiten translation of $M_j^{\la}$, and $x^{\coind {M_{j}^\la}} = x_{2\nu_n(i)-1}x_{2\nu_n(i)}^{-1}$.  
\end{prop}
\begin{proof}
We can see that $D\si_n(M_j^\la) \cong M_{\nu_n(j)}^\la$ by considering the action of $D\si_n$ on the basis $B_j^\la$ of $M_j^\la$.  Its image is a basis of $D\si_n(M_j^\la)$, whose elements we denote by 
\[ \{t_\la^*: t \in Q_0, 0 \leq \la \leq i-1, 2(i-\la)-1 \leq t \leq 2(n - \la).\]
The action of the edges of $Q_n$ on this basis are easily computed by taking duals and reindexing by $\si_n$.  The reader may check that
\[ t^*_\la \mapsto (\si_n t)_{n - \la - \floor*{\frac{t+1}{2}}} \]
induces a bijection between this basis and the basis $B_{\nu_n(j)}^\la$ of $M_{\nu_n(j)}^\la$ that is compatible with the action of the edges of $Q_n$, hence induces an isomorphism of modules.  That $M_{\nu_n(j)}^{\la} \cong \tau M_j^{\la}$ follows from checking easily that $I_{2j-1} \xrightarrow{\la} I_{2j}$ is also obtained by applying the Nakayama functor to $P_{2j} \xrightarrow{\la} P_{2j-1}$.
\end{proof}

\begin{remark}
The quiver $Q_n$ has a maximal green sequence obtained by mutating alternately at all odd-numbered vertices or all even-numbered vertices $n+1$ times \cite{Alim2011}.  This amounts to advancing $n+1$ steps in the $A_n$ $Q$-system \cite{Kedem2013}, hence the associated automorphism is discrete integrable in the sense that it preserves the Hamiltonians $H_i$ \cite{DiFrancesco,Gekhtman2011}.  On the other hand, by results of \cite{Keller2006} this automorphism has the property that, after composing with a specific permutation of the vertices of $Q_n$, it takes $CC(M)$ to $CC(\tau M)$ (this is the unique permutation inducing an isomorphism of the principally-framed quiver and the result of applying the maximal green sequence to the principally-framed quiver \cite{Brustle2012})\footnote{We thank Pierre-Guy Plamondon for pointing out the relevance of \cite{Keller2006} to us.}.  In our case, this permutation coincides with the one induced by the Nakayama involution $\nu_n$, which can be shown by studying the induced maximal green sequences of the $A_n$ subquivers of $Q_n$.  Thus the discrete integrability of (the $(n+1)$st iteration of) the $Q$-system is equivalent to the fact that $M_{\nu_n(j)}^{\la} \cong \tau M_j^{\la}$.  We note that the conservation of expectation values of Wilson lines in $\CN=2$ gauge theories by their monodromy operators is also treated from a representation-theoretic perspective in \cite{DelZotto}.
\end{remark}

\subsection{Relativistic Toda Hamiltonians are Cluster Characters}

In this section we compute the cluster characters of the modules $M_i^{\la}$ introduced in the previous section and show they coincide with the open relativistic Toda Hamiltonians $H_i$.  In particular, since $CC(M_{i}^\la$ is independent of the value of $\la$, the following theorem implies that the Hamiltonians are elements of the generic basis of 

\begin{thm}\label{thm:heqcc}
For each $1 \leq i \leq n$, we have $H_i = CC(M_{i}^\la)$ for all $\la \in \BP^1$.
\end{thm}

\begin{remark}
In particular, since $CC(M_{i}^\la$ is independent of the value of $\la$, the Hamiltonians are elements of the generic basis \cite{Dupont2011}.  This basis consists of cluster characters of generic representations of each possible index, and generalizes the dual semicanonical basis to arbitrary cluster algebras \cite{Geiss2012}.  It seems likely that when applied to the case at hand any reasonable construction of canonical bases for cluster algebras will also contain these Hamiltonians.  However, already in the case of the 2-Kronecker quiver different standard constructions will produce different bases for the subalgebra consisting of polynomials in $H_1$.  For example, the generic basis contains $H_1^2$, the triangular basis \cite{Berenstein2014} contains $H_1^2-1$, and the theta \cite{Gross2014} basis contains $H_1^2 - 2$.
\end{remark}

\begin{proof}
There are two components to the proof.  First, we prove that the coindex of $M_{i}^\la$ agrees with the corresponding term appearing in $H_i$.  Second, we construct a bijection between nonintersecting $i$-tuples of closed directed paths in $\ol{\CN}_{\mb{i}}$ and successor-closed subgraphs of $\Ga^\la_{\nu_n(i)}$, showing that this identifies weights of paths with dimension vectors in the appropriate sense.

Recall from \cref{prop:weightedsum} that $H_i$ is equal to the weighted sum
\[
\sum_{P = \{p_\ell\}_{\ell=0}^{i-1}} \wht(P)
\]
of all nonintersecting $i$-tuples of closed directed paths in $\ol{\CN}_{\mb{i}}$.  There is a unique such $i$-tuple $P_{\min}$ such that 
\[
\wht(P)/\wht(P_{\min}) \in \BC[y_1,\dotsc,y_{2n}]
\]
for each other $i$-tuple $P$.  It is straightforward to see that
\[
\wht(P_{\min}) = \prod_{\ell=0}^{i-1}\left(\prod_{j=2(n-\ell)+1}^{2n}y_j\right).
\]
Equivalently, $\wht(P_{\min})$ is the contribution to $H_i$ of the action of
\[
\prod_i(y_{2i-1}y_{2i})^{\om_i^\vee}
\]
on the lowest weight space of $\bigwedge^i \BC^{n+1}$ (which has weight $-\om_{\nu_n(i)}$).  
Since $y_{2i-1} = \prod_j x_{2j}^{C_{ij}}$ and $y_{2i} = \prod_j x_{2j-1}^{-C_{ij}}$, where $C$ is the $A_n$ Cartan matrix, we have 
\begin{align*}
%y_1^{\om_1^\vee}y_2^{\om_1^\vee}\cdots y_{2n-1}^{\om_n^\vee}y_{2n}^{\om_n^\vee} &= \prod_i(y_{2i-1}y_{2i})^{\om_i^\vee} \\
\prod_i(y_{2i-1}y_{2i})^{\om_i^\vee} &= \prod_{i,j} (x_{2j-1}^{-1}x_{2j})^{C_{ij}\om_i^\vee}\\
&= \prod_{j}(x_{2j-1}^{-1}x_{2j})^{\al_j^\vee}.
\end{align*}
But on the lowest weight space this acts by the scalar
\begin{align*}
%\prod_{j}(x_{2j-1}^{-1}x_{2j})^{\<\al_j^\vee|-\om_{i}\>} &= x_{2\nu_n(i)-1}x_{2\nu_n(i)}^{-1},
\prod_{j}(x_{2j-1}^{-1}x_{2j})^{-\om_{\nu_n(i)}(\al_j^\vee)} &= x_{2\nu_n(i)-1}x_{2\nu_n(i)}^{-1},
\end{align*}
hence
\[
\wht(P_{\min}) = x^{\coind {M_{i}^\la}}
\]
by \cref{prop:injectiveresolution}.

Recall from \cref{prop:Mbasis} that the vertices of $\Ga^\la_{i}$ are the elements of the basis 
\[
B_i^\la =\{ t_\ell : t \in Q_0, 0 \leq \ell \leq i-1, 2(i-\ell)-1 \leq t \leq 2(n - \ell)\}.
\]
Given a successor-closed subquiver $\Ga' \subset \Ga_i^\la$ and $0 \leq \ell \leq i-1$, let 
\[
m_\ell(\Ga') = \min\{t: t_\ell \in \Ga'_0\}
\]
if the right-hand side is nonempty, and $m_\ell(\Ga') = 2(n-\ell)+1$ otherwise.  
For $2(i-\ell)-1 \leq t, t' \leq 2(n - \ell)$, there is an arrow $t_\ell \to t'_\ell$ in $\Ga^\la_{i}$ if and only if $t' = t+1$.  In particular, we have
\[
\Ga'_0 = \{t_\ell \in B_i^\la: t \geq m_\ell(\Ga')\},
\]
so $\Ga'$ is completely determined by the $i$-tuple $\{m_\ell(\Ga')\}_{\ell = 0}^{i-1}$.  On the other hand, for $0 \leq \ell, \ell' \leq i-1$, $\ell \neq \ell'$, there is an arrow $t_\ell \to t'_{\ell'}$ in $\Ga^\la_{i}$ if and only if $\ell' = \ell+1$, $t$ is even, and $t' = t-3$.  It follows that an $i$-tuple $\{m_\ell\}_{\ell = 0}^{i-1}$ is of the form $\{m_\ell(\Ga')\}_{\ell = 0}^{i-1}$ for some successor-closed subquiver $\Ga'$ if and only if
\begin{equation}\label{eq:nonintersecting}
\begin{gathered}
2(i-\ell)-1 \leq m_\ell \leq 2(n - \ell)+1 \text{ for }0 \leq \ell \leq i-1,\\
m_{\ell+1} \leq m_\ell - 3 \text{ for $0 \leq \ell \leq i-2$ and $m_\ell$ even}\\
m_{\ell+1} \leq m_\ell - 2 \text{ for $0 \leq \ell \leq i-2$ and $m_\ell$ odd}.
\end{gathered}
\end{equation}

The term contributed to $CC(M_i^\la)$ by the associated submodule $N_{\Ga'}$  is
\[
y^{\dim N_{\Ga'}} = \prod_{\ell=0}^{i-1} \left(\prod_{j=m_{\ell(\Ga')}}^{2(n-\ell)}y_j\right).
\]

Now let $P=\{p_\ell\}_{\ell = 0}^{i-1}$ be a nonintersecting $i$-tuple of closed directed paths in $\ol{\CN}_{\mb{i}}$, indexed so that for $i < j$, $p_i$ lies below $p_j$ (equivalently, $\wht(p_j)/\wht(p_i)\in \BC[y_1,\dotsc,y_{2n}]$).  For $0 \leq \ell \leq i-1$, define $m_\ell = m_\ell(P)$ by the condition
\[
\wht(p_\ell) = y_{m_\ell}y_{m_\ell+1} \cdots y_{2n}
\]
if $\wht(p_\ell) \neq 1$, and $m_\ell(\Ga') = 2(n-\ell)+1$ if $\wht(p_\ell) = 1$.  It is straightforward to see that this assignment defines a bijection between the set of nonintersecting $i$-tuples of closed directed paths in $\ol{\CN}_{\mb{i}}$ and the set of $i$-tuples $\{m_\ell\}_{\ell = 0}^{i-1}$ satisfying \labelcref{eq:nonintersecting}, hence the set of successor-closed subquivers of $\Ga_i^\la$.  We have $m_\ell(P_{\min}) = 2(n-\ell)+1$, so that
\begin{align*}
\wht(P)/\wht(P_{\min}) &= \prod_{\ell=0}^{i-1} \left(\prod_{j=m_{\ell(\Ga')}}^{2(n-\ell)}y_j\right).
\end{align*}
But then
\begin{align*}
H_i & = \wht(P_{\min})\sum_{P = \{p_\ell\}}\wht(P)/\wht(P_{\min})\\
& = x^{\coind {M_{i}^\la}}\sum_{\Ga' \subset \Ga_i^\la}y^{\dim N_{\Ga'}} \\
& = CC(M_i^\la),
\end{align*}
completing the proof.
\end{proof}

\begin{example}\label{ex:H3}
On the left is the coefficient quiver $\Ga^\la_3$ of the 18-dimensional representation $M_3^\la$ of $J(Q_5,W_5)$ for $\la = (1:0)$.  On the right is the directed graph $\CN_{\mb{i}}$ associated with $SL_6^{c,c}$.  There are 61 successor-closed subquiver of $\Ga^\la_3$, hence 61 submodules of $M_3^\la$.  These correspond to 61 nonintersecting triples of closed directed paths in $\ol{\CN}_{\mb{i}}$.  The circled subquiver on the left corresponds to the highlighted triple on the right, and to a 6-dimensional submodule of $M_3^\la$.

\[
\begin{tikzpicture}[thick,>=stealth']
\newcommand*{\ky}{.9}
\newcommand*{\dy}{.5}
\newcommand*{\dx}{1.5}
\node [matrix] (socle) at (0,0)
{
\node (6a) at (0,0) {$5_0$};
\node (5a) at (0,-\ky) {$6_0$};
\node (6b) at (0,-2*\ky-2*\dy) {$5_1$};
\node (5b) at (0,-3*\ky-2*\dy) {$6_1$};
\node (6c) at (0,-4*\ky-4*\dy) {$5_2$};
\node (5c) at (0,-5*\ky-4*\dy) {$6_2$};
\node (8a) at (\dx,-\ky-\dy){$7_0$};
\node (7a) at (\dx,-2*\ky-\dy){$8_0$};
\node (8b) at (\dx,-3*\ky-3*\dy){$7_1$};
\node (7b) at (\dx,-4*\ky-3*\dy){$8_1$};
\node (10a) at (2*\dx,-2*\ky-2*\dy){$9_0$};
\node (9a) at (2*\dx,-3*\ky-2*\dy){$10_0$};
\node (4a) at (-\dx,-\ky-\dy){$3_1$};
\node (3a) at (-\dx,-2*\ky-\dy){$4_1$};
\node (4b) at (-\dx,-3*\ky-3*\dy){$3_2$};
\node (3b) at (-\dx,-4*\ky-3*\dy){$4_2$};
\node (2a) at (-2*\dx,-2*\ky-2*\dy){$1_2$};
\node (1a) at (-2*\dx,-3*\ky-2*\dy){$2_2$};

\newcommand{\buf}{.55};
\draw [blue] ($(5c)+(\buf,0)$) 
to[out=90,in=-90] ($(7b)+(\buf,0)$)
to[out=90,in=-90] ($(9a)+(\buf,0)$)
to[out=90,in=0] ($(9a)+(0,\buf)$)
%to[out=180,in=90] ($(9a)+(-.4,0)$)
to[out=180,in=20] ($(9a)+(-.8,.2)$)
to[out=20+180,in=90] ($(8b)+(-\buf,0)$)
to[out=-90,in=0] ($(6c)+(0,\buf)$)
to[out=180,in=-90] ($(4b)+(\buf,0)$)
to[out=90,in=0] ($(4b)+(0,\buf)$)
to[out=180,in=90] ($(4b)+(-\buf,0)$)
to[out=-90,in=90] ($(3b)+(-\buf,0)$)
to[out=-90,in=90] ($(5c)+(-\buf,0)$)
to[out=-90,in=180] ($(5c)+(0,-\buf)$)
to[out=0,in=-90] ($(5c)+(\buf,0)$);

\draw [->] (2a) to (1a);
\draw [->] (4a) to (3a);
\draw [->] (4b) to (3b);
\draw [->] (6a) to (5a);
\draw [->] (6b) to (5b);
\draw [->] (6c) to (5c);
\draw [->] (8a) to (7a);
\draw [->] (8b) to (7b);
\draw [->] (10a) to (9a);
\draw [->] (5a) to (8a);
\draw [->] (5a) to (4a);
\draw [->] (5b) to (8b);
\draw [->] (5b) to (4b);
\draw [->] (7a) to (10a);
\draw [->] (7a) to (6b);
\draw [->] (3a) to (2a);
\draw [->] (3a) to (6b);
\draw [->] (1a) to (4b);
\draw [->] (9a) to (8b);
\draw [->] (3b) to (6c);
\draw [->] (7b) to (6c);\\
};
\newcommand*{\Xh}{0.45}
\newcommand*{\ra}{0} % row a height
\newcommand*{\rb}{1} % row b height
\newcommand*{\rc}{2} % row c height
\newcommand*{\length}{4.45} % length of network
\newcommand*{\vla}{.5} % line a position
\newcommand*{\vlb}{1} % line b position
\newcommand*{\vlc}{1.5} % line c position
\newcommand*{\vld}{2} % line d position
\newcommand*{\varrowpos}{.5} % position of arrows on vertical lines
\node [matrix,cells={scale=1.2}] (network) at (8,0) {

\foreach \x in {0,1,2,3,4} {
\draw (.5+\x*.5,5-\x) -- (.5+\x*.5,5-\x-1);
\draw [->] (.5+\x*.5,5-\x-1) -- (.5+\x*.5,5-\x-.46);
\draw (.5+\x*.5+1,5-\x) -- (.5+\x*.5+1,5-\x-1);
\draw [->] (.5+\x*.5+1,5-\x) -- (.5+\x*.5+1,5-\x-.54);
};

\foreach \y in {0,1,2,3,4,5} \draw [<-<] (0,\y) -- (\length,\y);

\newcommand*{\thck}{1.2pt};
\draw [blue,line width=\thck,>->] (\length,4) -- (0,4);
\draw [blue,line width=\thck,>->] (\length,2) -- (0,2);
\draw [blue,line width=\thck,>->] (\length,1) -- (3.5,1) -- (3.5,0) -- (2.5,0) -- (2.5,1) -- (0,1);
\draw [blue,line width=\thck,->] (3.5,1) -- (3.5,.46);
\draw [blue,line width=\thck,->] (2.5,0) -- (2.5,.54);

\foreach \x/\y/\i in {0/4/1,1/4/2,.5/3/3,1.5/3/4,1/2/5,2/2/6,1.5/1/7,2.5/1/8,2/0/9,3/0/10}
\node at (1+\x,\Xh+\y) {$y_{\i}$};\\
};

\end{tikzpicture}
\]
\end{example}

\subsection{Cluster Characters and Framed Quiver Moduli}

There is a general correspondence between quiver Grassmannians and moduli spaces of framed quiver representations \cite{Reineke}, and in this section we recall how to rewrite the cluster characters $CC(M_i^\la)$ in terms of framed quiver moduli.  While the perspective of quiver Grassmannians is more natural from the point of view of cluster algebras \cite{Caldero2004,Caldero2006,Palu2008}, the perspective of framed quiver moduli is more natural from the point of view of framed BPS indices in $\CN=2$ field theory \cite{Gaiotto2010,Chuang2013,Cordova2013,Cirafici2013}.  Thus our aim here is to clarify how to equate the expressions $CC(M_k^\la)$ with the types of expressions considered in the literature on line operators in $\CN=2$ theories.  

A framing of a quiver $Q$ is a new quiver $\ol{Q}$ containing $Q$ as a full subquiver along with one additional vertex called the framing vertex.  We generally assume a fixed labeling of $Q_0$ by $\{1,\dotsc,n\}$, and extend this to label the framing vertex by $0$.  If $Q$ comes with a potential $W$, we will denote by $\ol{W}$ a choice of its extension to $\ol{Q}$; that is, $\ol{W}$ is a sum of $W$ and some collection of cycles which include arrows whose source is the framing vertex.

\begin{defn} A framed representation of $J(Q,W)$ is a representation $M$ of $J(\ol{Q},\ol{W})$ for some framed quiver with potential $(\ol{Q},\ol{W})$, and such that $\dim M_0 = 1$.  We refer to the subspace $\oplus_{i=1}^n M_i$ as the unframed part of $M$.  
\end{defn}

Recall that a stability condition on a finite-length abelian category $\CC$ is a homomorphism $Z: K_0(\CC) \to \BC$ such that $Z([M])$ lies in the upper-half plane for any object $M$ of $\CC$.  An object $M$ is stable if
\[
\arg Z([L]) < \arg Z([M]) 
\]
for all proper subobjects $L$ of $M$.  Given a stability condition on $J(\ol{Q},\ol{W})\dmod$ and a dimension vector $e \in \BZ^n$, we write $\CM^{s}_{e}(\ol{Q},\ol{W})$ for the moduli space of stable framed representations whose unframed part has dimension $e$.

We say a stability condition $Z$ on $J(\ol{Q},\ol{W})\dmod$ is cocyclic if
\[
\arg Z([S_0]) < \arg Z([S_i])
\]
for all $1 \leq i \leq n$.  We say a framed representation $M$ of $J(Q,W)$ is cocyclic if every nontrivial submodule contains $M_0$.  It is straightforward to show that a framed representation of $J(Q,W)$ is stable with respect to a cocyclic stability condition if and only if it is cocyclic.

For the quiver $Q_n$, we write $\ol{Q}_n^k$ for the framed quiver with no oriented 2-cycles such that
\[
(\ol{Q}_n^k)_{0i} = \begin{cases} 1 & i = 2k-1\\ -1 &i = 2k\\ 0 &\text{otherwise}.\end{cases}
\]
For $\la = (\la_1 : \la_2) \in \BP^1$, let $\ol{W}_n^\la$ denote the potential
\[
\ol{W}_n^\la = W_n + \la_1 a_kcd + \la_2 b_kcd
\]
on $\ol{Q}_n^k$, where $c$ and $d$ denote the arrows from $0$ to $2k-1$ and from $2k$ to $0$, respectively.  

\begin{prop}\label{thm:framedpotential}
Let $Z$ be a cocyclic stability condition on $J(\ol{Q}_n^k,\ol{W}_n^\la)\dmod$, and $M$ a cocylic framed representation.  The unframed part of $M$ is isomorphic to a unique subrepresentation of $M_{\nu_n(k)}^\la$. In particular,  $\CM^{s}_{e}(\ol{Q}_n^k,\ol{W}_n^\la) \cong \Gr_e M_{\nu_n(k)}^\la$ for all dimension vectors $e \in \BN^{2n}$, so
\[
CC(M_{\nu_n(k)}^\la) = x_{2k-1}x_{2k}^{-1}\sum_{e \in \BZ^n} \chi(\CM^{s}_{e}(\ol{Q}_n^k,\ol{W}_n^\la)) y^e.
\]
\end{prop}
\begin{proof}
Since $M$ is cocyclic, it is a submodule of $I_0$ with $\dim M_0 =1$ hence $c$ acts by $0$.  The unframed part $M^u$ of $M$ is then a $J(Q_n,W_n)$-module which injects into $I_{2k}$.  But the relation $\del_{c} W^\la = (\la_1 a + \la_2 b)d = 0$ forces the image of $M^u$ to lie in the kernel of the map $I_{2k} \xrightarrow{\la} I_{2k-1}$, which by \cref{prop:injectiveresolution} is $M_{\nu_n(k)}^\la$.  Since each $\Gr_e M_{\nu_n(k)}^\la$ is either a point or empty, we trivially obtain the stated isomorphism.
\end{proof}

\section{Relativistic Toda Systems and Spectral Networks}\label{sec:reltodaandnets}

In this section we identify the phase space $SL_{n+1}^{c,c}/\Ad H$ with a wild character variety of $\BP^1$ with singularities at $0$ and $\infty$, and the open relativistic Toda Hamiltonians with traces of holonomies around the nontrivial cycle of $\BC^*$.  This character variety is related by the wild nonabelian Hodge correspondence to the periodic nonrelativistic Toda system, viewed as a meromorphic Hitchin system on $\BP^1$.  The spectral networks of the periodic Toda system endow it with a structure essentially that of a cluster variety of type $Q_n$; more precisely, it is birational to a space intermediate between the corresponding $\CA$- and $\CX$-varieties.  This implicitly identifies $SL_{n+1}^{c,c}/\Ad H$ with the wild character variety up to a finite cover, and showing this takes traces of holonomies to Hamiltonians consists in carefully studying the trajectories of certain differential equations defined by the periodic Toda spectral curves.  We show that the description of the relativistic Toda Hamiltonians as weighted sums of paths in a directed graph reappears exactly in the path-lifting description of the parallel transport around $\BC^*$, the directed graph  $\ol{\CN}_{\mb{i}}$ used in \cref{sec:paths} to describe the cluster coordinates on $SL_{n+1}^{c,c}/\Ad H$ reemerging here as the 1-skeleton of the periodic Toda spectral curve.

\subsection{Path-lifting in Nonsimple Spectral Networks}\label{sec:nonsimple}

For a generic spectral curve of even a simple Hitchin system, it may be difficult to explicitly describe the trajectories of differential equations that comprise the associated spectral network.  Our purposes, however, give us the flexibility to restrict out attention to special Toda spectral curves with additional global symmetries.  This simplifies the situation enough to perform an essentially exact analysis, but these nongeneric curves will no longer have only simple ramification.  Thus in this section we collect some technical preliminaries concerning the path-lifting rules defined by nonsimple spectral networks.

Such a network can be ``resolved'' into a simple spectral network in several ways, each of which defines an equivalent path-lifting rule in a suitable sense.  In particular, suppose we have a path $\wp$ that passes through a neighborhood of a nonsimple branch point, for example in the case of $\stdknet$ a path that is a slight perturbation of a line through the origin.  The number of walls crossed by $\wp$ grows quadratically with $N$.  However, we can construct a simple network $\stdknet'$ equivalent to $\stdknet$ such that $\wp$ crosses only $N-1$ walls of $\stdknet'$.  In particular, $\stdknet'$ provides a much less redundant description of the parallel transport along $\wp$.  Analytically, if a nonsimple network is defined by a spectral curve $\Si$ in the cotangent bundle of a Riemann surface $C$, resolutions of this sort may be constructed from a generic perturbation of $\Si$.  Below we provide an explicit combinatorial description of all possible resolutions, which will prove useful for practical computations later.

Recall the network $\stdknet$ from \cref{sec:snetapp} and \cref{fig:stdknet}.  The points along any fixed wall $p_c$ of $\stdknet$ have the same argument, which we denote by $\arg p_c$.

\begin{defn}
Let $p_c$ be a wall of $\stdknet$ for some $N>2$ and let $\th = \arg p_c$.  The splitting of $\stdknet$ of phase $\th$ is the subset $S_\th \coloneqq \{p_c: \arg p_c \in [\th,\th+\pi)\}$ of walls of $\stdknet$.  A splitting of $\stdknet$ is a subset of its walls of this form for some phase
\end{defn}

\begin{defn}\label{def:splitting}
Let $\fb$ be a branch point of $\Si \to C$ whose monodromy is an $N$-cycle for $N>2$, and $\CW$ a spectral network subordinate to $\Si$.  By assumption, there is a bijection between the walls of $\CW$ born at $\fb$ and the walls of $\stdknet$.  A splitting of $\CW$ at $\fb$ is a subset $S = \{p_c\}$ of these walls which is identified with a splitting of $\stdknet$ under this bijection.
\end{defn}

We will identify the set of ordered pairs of distinct indices $0,1,\dotsc,N-1$ with roots of $SL_N$ in the natural way.  In particular, if the sheets of a branched cover $\Si$ are labeled locally by $0,1,\dotsc,N-1$, the walls of a spectral network subordinate to $\Si$ are then labeled by roots of $SL_N$.

\begin{defn}
We associate a choice $\Phi$ of positive roots of $SL_N$ to the splitting $S_\th$ of $\stdknet$ as follows.  First choose a branch cut along a ray $\{z: \arg z = \vphi\}$ for some $\vphi \nin [\th,\th+\pi)$ and a labeling of the sheets of $\Si_N$ by $\{0,1,\dotsc,N-1\}$ such that the counterclockwise monodromy around $0$ is the $N$-cycle
\[
(01\cdots(N-1)).
\]
Then $\Phi$ is the set of ordered pairs $ij$ for which there is a wall $p_c$ with outwardly oriented label $ij$ and with $\arg p_c \in [\th+\frac{\pi}{N+1},\th+\pi)$.  Note that the definition of $\Phi$ depends on the choice of labeling.  Let $\De_a$ and $\De_b$ be the sets of positive roots for which the corresponding wall has argument $\th+\frac{N\pi}{N+1}$ or $\th+\frac{\pi}{N+1}$, respectively. In other words, $\De_a$ and $\De_b$ are the labels of the walls of $S_\th$ with the largest and the second smallest argument, respectively.  The subset $\De \subset \Phi$ of simple positive roots is exactly $\De_a \cup \De_b$.  
\end{defn}

\begin{defn}\label{def:splittingdata}
To a splitting $S$ of a general spectral network $\CW$ at some branch point $\fb$, and we associate choices $\Phi$ and $\De = \De_a \cup \De_b$ of positive and simple positive roots via its the local equivalence with $\stdknet$.  This again depends on a branch cut and suitable labeling of the sheets near $\fb$.  We let $\si$ be the unique permutation of $\{0,1,\dotsc,N-1\}$ such that 
\[\De = \{\si_0\si_1,\si_1\si_2,\dotsc,\si_{N-2}\si_{N-1}\}.
\]
\end{defn}

\begin{comment}
\begin{figure}
\caption{An example of $\De$ for some splitting.}
\label{fig:simpleroots}
\end{figure}
\end{comment}

To the splitting $S_\th$ of $\stdknet$ we associate a new spectral network $\stdknet'$, subordinate to a branched cover $\Si'$ of $\BC$ with only simple ramification defined as follows.  There is a simple branch point $\fb_{ij}$ inside the unit disk for each $ij \in \De$.  If $ij \in \De_a$ (resp. $ij \in \De_b$) then $\arg \fb_{ij} = \th+\frac{\pi}{4}$ (resp. $\arg \fb_{ij} = \th+\frac{5\pi}{4}$), and if $ij$ and $i'j'$ are both in $\De_a$ or both in $\De_b$, then $|\fb_{ij}| < |\fb_{i'j'}|$ if and only if $(\si^{-1})_{i'} < (\si^{-1})_{i}$.  We choose $N-1$ branch cuts that do not intersect in the unit disk but coalesce onto the fixed branch cut of $\stdknet$ outside the unit disk.  The walls of $\stdknet$ divide the unit circle into $2(N+1)$ disjoint arcs, and these $N-1$ branch cuts should each cross the unit circle on the same arc.  Then $\Si'$ is the branched cover whose sheets are labeled by $\{0,1,\dotsc,N-1\}$ over the complement of the branch locus so that the monodromy around $\fb_{ij}$ interchanges $i$ and $j$. 

\begin{figure}
\begin{tikzpicture}[decoration={snake,amplitude=1,segment length=7}]
\newcommand*{\rad}{5};
\newcommand*{\off}{2};
\newcommand*{\din}{1};
\newcommand*{\sk}{3};
%\draw (0,0) circle (\rad);
\foreach \n in {0,1,...,11} {\draw[-stealth'] (15+30*\n:\rad) to (15+30*\n:\rad+.4);};

\node [cross out,draw=orange,line width=.6mm,rounded corners] (b0) at (225:.66) {};
\node [cross out,draw=orange,line width=.6mm,rounded corners] (b1) at (225+180:.66) {};
\node [cross out,draw=orange,line width=.6mm,rounded corners] (b2) at (225:2) {};
\node [cross out,draw=orange,line width=.6mm,rounded corners] (b3) at (225+180:2) {};

\node (bcm) at (92:\rad-\din+.3) {};
\draw [decorate,color=orange] (b0) to[out=120,in=70+180] (120:1.8) to[out=70,in=-90] (bcm);
\draw [decorate,color=orange] (b1) to[out=90,in=-90] (bcm);
\draw [decorate,color=orange] (b2) to[out=120,in=50+180] (120:2.5) to[out=50,in=-90]  (bcm);
\draw [decorate,color=orange,shorten >=0] (b3) to[out=90,in=-90] (bcm);
\draw [decorate,color=orange,shorten <=0] ($(bcm)-(0,.13)$) to[out=90,in=-90] ($(bcm)+(0,1.2)$);

\node [color=orange] at ($(bcm)+(0,1.4)$) {(01234)};

\foreach \n in {0} {

\draw [name path=p12] (b0)
to[out=0+180*\n,in=-15-\sk-4+180+180*\n] (-15+\off+180*\n:\rad-\din) 
to[out=-15-\sk-4+180*\n,in=-15+180+180*\n] (-15+180*\n:\rad);
\node at (-15+\off+4+180*\n:\rad-\din) {12};

\draw [name path=p03] (b2) 
to[out=0+180*\n,in=-15+180+180*\n] (-15+180*\n:\rad);
\node at (-15-\off-3+180*\n:\rad-\din) {03};

\draw (b0)
to[out=180+180*\n,in=195+\sk+4+180+180*\n] (195-\off+180*\n:\rad-\din) 
to[out=195+\sk+4+180*\n,in=195+180+180*\n] (195+180*\n:\rad);
\node at (195-\off-4+180*\n:\rad-\din) {21};

\draw (b2)
to[out=180+180*\n,in=195+180+180*\n] (195+180*\n:\rad);
\node at (195+\off+4+180*\n:\rad-\din) {30};

\draw [name path=p20] (b1) to[out=-90,in=45] (-110:1.3)
to[out=225+\sk+180*\n,in=225+180+180*\n] (225+180*\n:\rad);
\node at (225-3+180*\n:\rad-\din) {20};

\draw [name path=p34] (b3) to[out=-90,in=45] (-70:1.3) 
to[out=-135+180*\n,in=225-\sk-25+180+180*\n] (225+\off+3+180*\n:\rad-\din) 
to[out=225-\sk-25+180*\n,in=225+180+180*\n] (225+180*\n:\rad);
\node at (225+\off+8+180*\n:\rad-\din) {34};

\draw [name intersections={of=p03 and p20,by=int1}];
\draw [name intersections={of=p03 and p34,by=int2}];
\draw [name intersections={of=p12 and p20,by=s10}, name path=p10] (s10) 
to[out=-45+180*\n,in=45+180*\n] ($(int1)!.5!(int2)$) 
to[out=45+180+180*\n,in=40+180*\n] (-105:1.7)
to[out=40+180+180*\n,in=255+180+\sk+5+180*\n] (255-\off-1+180*\n:\rad-\din) 
to[out=255+\sk+5+180*\n,in=75+180*\n] (255+180*\n:\rad);
\node at (255-\off-5+180*\n:\rad-\din) {10};

\draw [name intersections={of=p10 and p03,by=s13}, name path=p13] (s13)
to[out=320+180*\n,in=315+\sk+15+180+180*\n]  (315-4+180*\n:\rad-\din) 
to[out=315+\sk+15+180*\n,in=315+180+180*\n] (315+180*\n:\rad);
\node at (315-8+180*\n:\rad-\din) {13};

\draw [name intersections={of=p13 and p34,by=s14}, name path=p14] (s14)
to[out=285+180*\n,in=285+180+180*\n] (285+180*\n:\rad);
\node at (285+3+180*\n:\rad-\din) {14};

\draw [name intersections={of=p03 and p20,by=s23}, name path=p23] (s23)
to[out=285+180*\n,in=285+\sk+20+180+180*\n]  (285-\off-2+180*\n:\rad-\din) 
to[out=285+\sk+20+180*\n,in=285+180+180*\n] (285+180*\n:\rad);
\node at (285-\off-7+180*\n:\rad-\din) {23};

\draw [name intersections={of=p23 and p34,by=s24}, name path=p24] (s24) 
to[out=-100+180*\n,in=255+180-\sk-2+180*\n] (255+\off+180*\n:\rad-\din) 
to[out=255-\sk-2+180*\n,in=75+180*\n] (255+180*\n:\rad);
\node at (255+\off+4+180*\n:\rad-\din) {24};

\draw [name intersections={of=p03 and p34,by=s04}, name path=p04] (s04) 
to[out=325+180*\n,in=315+180+180*\n] (315+180*\n:\rad);
\node at (315+3+180*\n:\rad-\din) {04};
};

\foreach \n in {1} {

\draw [name path=pu31] (b1)
to[out=0+180*\n,in=-15-\sk-4+180+180*\n] (-15+\off+180*\n:\rad-\din) 
to[out=-15-\sk-4+180*\n,in=-15+180+180*\n] (-15+180*\n:\rad);
\node at (-15+\off+4+180*\n:\rad-\din) {31};

\draw [name path=pu40] (b3) 
to[out=0+180*\n,in=-15+180+180*\n] (-15+180*\n:\rad);
\node at (-15-\off-3+180*\n:\rad-\din) {40};

\draw (b1)
to[out=180+180*\n,in=195+\sk+2+180+180*\n] (195-\off+180*\n:\rad-\din) 
to[out=195+\sk+2+180*\n,in=195+180+180*\n] (195+180*\n:\rad);
\node at (195-\off-4+180*\n:\rad-\din) {02};

\draw (b3)
to[out=180+180*\n,in=195+180+180*\n] (195+180*\n:\rad);
\node at (195+\off+4+180*\n:\rad-\din) {43};

\draw [name path=pu12] (b0) 
to[out=-90+180*\n,in=45+180*\n] (-110+180*\n:1.3)
to[out=225+\sk+180*\n,in=225+180+180*\n] (225+180*\n:\rad);
\node at (225-3+180*\n:\rad-\din) {01};

\draw [name path=pu03] (b2) 
to[out=-90+180*\n,in=45+180*\n] (-70+180*\n:1.3) 
to[out=-135+180*\n,in=225-\sk-25+180+180*\n] (225+\off+3+180*\n:\rad-\din) 
to[out=225-\sk-25+180*\n,in=225+180+180*\n] (225+180*\n:\rad);
\node at (225+\off+8+180*\n:\rad-\din) {42};

\draw [name intersections={of=pu31 and pu03,by=intu1}];
\draw [name intersections={of=pu40 and pu03,by=intu2}];
\draw [name intersections={of=pu31 and pu12,by=su32}, name path=pu32] (su32) 
%to[out=-45+180*\n,in=-30] ($(intu1)!.5!(intu2)$) 
%to[out=45+180+180*\n,in=40+180*\n] (-105+180*\n:1.7)
%to[out=40+180+180*\n,in=-45+180+\sk+5+180*\n] (-45-\off-1+180*\n:\rad-\din) 
to[out=135,in=-45] (-45+180*\n:\rad);
\node at (-45-\off-3+180*\n:\rad-\din) {32};

\draw [name intersections={of=pu32 and pu03,by=su02}, name path=pu02] (su02)
to[out=-95+180*\n,in=-105-\sk-24+180+180*\n]  (-105+7+180*\n:\rad-\din) 
to[out=-105-\sk-24+180*\n,in=-105+180+180*\n] (-105+180*\n:\rad);
\node at (-105+12+180*\n:\rad-\din) {41};

\draw [name intersections={of=pu31 and pu03,by=su01}, name path=pu01] (su01)
to[out=135,in=105+180] (105:\rad);
\node at (105+7:\rad-\din) {01};

\draw [name intersections={of=pu40 and pu01,by=su41}, name path=pu41] (su41)
to[out=145,in=-45-\sk-11+180+180*\n]  (-45+\off+2+180*\n:\rad-\din) 
to[out=-45-\sk-11+180*\n,in=-45+180+180*\n] (-45+180*\n:\rad);
\node at (-45+\off+6+180*\n:\rad-\din) {41};

\draw [name intersections={of=pu02 and pu40,by=su42}, name path=pu42] (su42) 
to[out=125,in=-75+180-\sk+1+180*\n] (-75-\off-1+180*\n:\rad-\din) 
to[out=-75-\sk+1+180*\n,in=-75+180+180*\n] (-75+180*\n:\rad);
\node at (-75-\off-5+180*\n:\rad-\din) {42};

\draw [name intersections={of=pu40 and pu03,by=su43}, name path=pu43] (su43) 
to[out=80,in=-105] (-105+180*\n:\rad);
\node at (-105-2+180*\n:\rad-\din) {32};
};

\end{tikzpicture}
\caption{The simple network $\CW'_{5}$ associated with the splitting $S_{\frac{13\pi}{12}}$ of $\CW_5$  consisting of walls lying in the lower-half plane.   See \cref{fig:stdknet} for an illustration of $\CW_5$.  Here $\De_a = \{03,12\}$, $\De_b = \{20,34\}$, and $\si = (012)$.  If $\wp$ is a path that follows the real axis, the parallel transport associated with $F(\wp,\CW'_5)$ factors as a product of elements of 4 simple root subgroups corresponding to the 4 walls $\wp$ intersects.}
\label{fig:resolvednet}
\end{figure}

The network $\stdknet'$ is determined up to isotopy by the following conditions:
\begin{enumerate} 
\item The intersections of $\stdknet$ and $\stdknet'$ with the complement of the unit disk coincide.  The walls of $\stdknet'$ are in bijection with those of $\stdknet$, though we abuse our terminology slightly: we will refer to as an $ij$-wall the union of a sequence of $ij$-walls, each consecutive pair of which meets at a joint.  For example, in this sense the middle picture of \cref{fig:junction} consists of three walls: an $ij$-wall crossing a $jk$-wall and giving birth to an $ik$-wall.  In the interior of the unit disk, no two walls of $\stdknet'$ cross more than once.  Each wall $p_c$ of $\stdknet'$ meets the unit circle exactly once and we write $\arg p_c$ for the argument of this point.
\item If $ij \in \De_a$ (resp. $\De_b$), the three walls of $\stdknet'$ born at $\fb_{ij}$ have arguments $\th$, $\th+\frac{N}{N+1}\pi$, $\th+\frac{N+2}{N+1}\pi$ (resp. $\th+\frac{1}{N+1}\pi$, $\th+\pi$, $\th+\frac{2N+1}{N+1}\pi$).  No walls of the same argument intersect inside the unit disk.  The walls $p_c$ with $\th+\frac{N+2}{N+1}\pi \leq \arg p_c \leq \th$ (resp. $\th+\frac{1}{N+1}\pi \leq \arg p_c \leq \th+\pi$) do not intersect the semidisk $\{z: |z| < 1, \th+\frac{1}{N+1}\pi < \arg z < \th+\frac{N+2}{N+1}\pi\}$ (resp. $\{z: |z| < 1, \th+\frac{N+2}{N+1}\pi < \arg z < \th+\frac{1}{N+1}\pi\}$).
\item The walls $p_c$ with $\th + \frac{1}{N+1}\pi < \arg p_c < \th + \frac{N}{N+1}\pi$ are not born at branch points, but rather at joint where two other walls cross; we call these secondary walls and the walls born at branch points primary walls.  The labels of the secondary walls are in bijection with $\Phi \smallsetminus \De$.   For some $0 < R < 1$, all joints where a primary wall crosses another wall lie in the region $|z|<R$, and all joints where two secondary walls cross lie in the region $|z| > R$.  A secondary $ij$-wall $p_c$ is born at a joint where an $ik$-wall crosses a $kj$-wall for $k$ satisfying $(\si^{-1})_k = 1+(\si^{-1})_i$.  This $ij$-wall crosses a primary $i'j'$-wall $p'_c$ if and only if $(\si^{-1})_{i'}\geq (\si^{-1})_j$ and $\arg p_c' \in \{\th+\frac{1}{N+1}\pi,\th+\frac{N}{N+1}\pi\}$.  Following $p_c$ away from its birthplace, $p_c$ intersects these primary walls in order of increasing $(\si^{-1})_{i'}$.  
\item The walls $p_c$ with $\th + \frac{N+2}{N+1}\pi < \arg p_c < \th+2\pi$ are also born at joints rather than branch points.  Their configuration is determined in the same fashion as those in the previous step.
\end{enumerate}

We let the reader convince themselves there exists a unique-up-to-isotopy spectral network satisfying the above conditions.

To compare the path-lifting rules defined by $\stdknet$ and $\stdknet'$ we define a 1-parameter family of branched covers $\Si_t$ with $\Si_0 = \Si$ and $\Si_1 = \Si'$.  The branch points of $\Si_t$ are $\fb_{ij,t} \coloneqq t \fb_{ij}$.  The family $\Si_t$ lets us identify homotopy classes of paths in $\Si$ and $\Si'$ with fixed endpoints in the complement of the unit disk.  

\begin{prop}
If $\wp$ is an open path with endpoints in the complement of the unit disk,
\[
\bF(\ti \wp,\stdknet) = \bF(\ti \wp,\stdknet'),
\]
where we identify homotopy classes of open paths in $\Si$ and $\Si'$ via the family $\Si_t$.
\end{prop}

\begin{proof}
The statement is clear from the definition of $\stdknet'$ when $\wp$ stays within a small neighborhood of the unit circle.  But this is sufficient since $\bF(\ti \wp,\stdknet)$ is clearly invariant when we alter $\wp$ by dilating the interior of the unit disk.  
\end{proof}

\begin{corollary}\label{cor:simpleres}
Any nonsimple spectral network $\CW$ defines a path-lifting rule equivalent to that of a simple network $\CW'$.
\end{corollary}

\begin{proof}
Since $\CW$ is equivalent to some $\stdjnet$ in a neighborhood of any branch point $\fb$, we can ``glue in'' the network $\stdjnet'$ associated with some splitting to obtain a new spectral network equivalent to $\CW$.  Doing this at every branch point with nonsimple ramification, we obtain a simple spectral network which is equivalent to $\CW$.
\end{proof}

\subsection{Toda Spectral Networks at Strong Coupling}\label{sec:todanetworks}

In this section we study in detail certain spectral networks of the periodic (nonrelativistic) Toda system, and use them to compute traces of holonomies on the associated wild character variety.  We view this system as a meromorphic Hitchin system on $\BP^1$ with singularities at $0$ and $\infty$ whose base is
\[
\CB_N \coloneqq \{(u_2(\frac{dz}{z})^2,\dotsc,u_{N-1}(\frac{dz}{z})^{N-1},(z+u_N+z^{-1})(\frac{dz}{z})^N: (u_2,\dotsc,u_N) \in \BC^{N-1}\}.
\]
The spectral curve $\Si_u \subset T^*\BC^*$ associated with a generic point $u \in \CB_N$ is a twice-punctured genus $N-1$ hyperelliptic curve, with cyclic monodromy around $0$ and $\infty$.  For generic $u$, $\Si_u$ has $2(N-1)$ simple branch points over $\BC^*$, and at $u_0 = (0,\dotsc,0)$ these coalesce into a pair of branch points at $\pm i$ with $N$-cyclic monodromy.  For our purposes it suffices to focus entirely on the spectral networks associated with $u_0$, and from now on we write $\Si$ for $\Si_{u_0}$ and $\ol{\Si}$ for the corresponding smooth projective curve.  In the language of Seiberg-Witten theory, we restrict our attention to the strong-coupling region of the Coulomb branch of pure $\CN=2$ $SU(N)$ gauge theory.  

To label the sheets of $\Si$ we introduce a pair of branch cuts running along the imaginary axis from $i$ to $\infty$ and from $-i$ to $0$.  We write $\om = e^{2\pi i/N}$, $\om_{ij} = \om^i - \om^j$, and label the sheets of $\Si$ by $\{0,1,\dotsc,N-1\}$.  The $k$th sheet is given by 
\[\la_k \coloneqq \frac{\om^k (z+z^{-1})^{1/N}}{z} dz,\]
where $(z+z^{-1})^{1/N}$ is taken to mean the branch with positive real values on $\BR_+$. As one crosses either branch cut from the right-half plane to the left half-plane the sheets are permuted by the $N$-cycle $(01\cdots(N-1))$.

Let $\CW_\th$ denote the spectral network of phase $\th$ subordinate to $\Si$.  The following lemma is our main source of qualitative control over the structure of $\CW_\th$, generalizing the $SU(2)$ analysis at the end of \cite{Klemm1996}.  In particular, the condition on $z'(t)$ in the statement applies to any parametrization of a wall of $\CW_\th$, and ultimately guarantees that they either lie entirely on the unit circle or never intersect it.

\begin{lemma}\label{lem:walls}
Fix $R>0$ and a phase $\phi$, and let $z(t)$ be a path in $\BC^*$ defined for $0 \leq t \leq R$.  Suppose that $z(t)$ satisfies
\[
\arg z'(t) = \phi + \arg \frac{z}{(z+z^{-1})^{1/N}}
\]
for $0 < t \leq R$, the value of $(z+z^{-1})^{1/N}$ being determined by analytic continuation along $z(t)$.  Suppose that for all sufficiently small $t > 0$, $|z(t)|$ is either less than $1$ and decreasing, equal to $1$ and constant, or greater than $1$ and increasing.  Then the same condition holds for all $t$.
\end{lemma}

\begin{figure}
\includegraphics{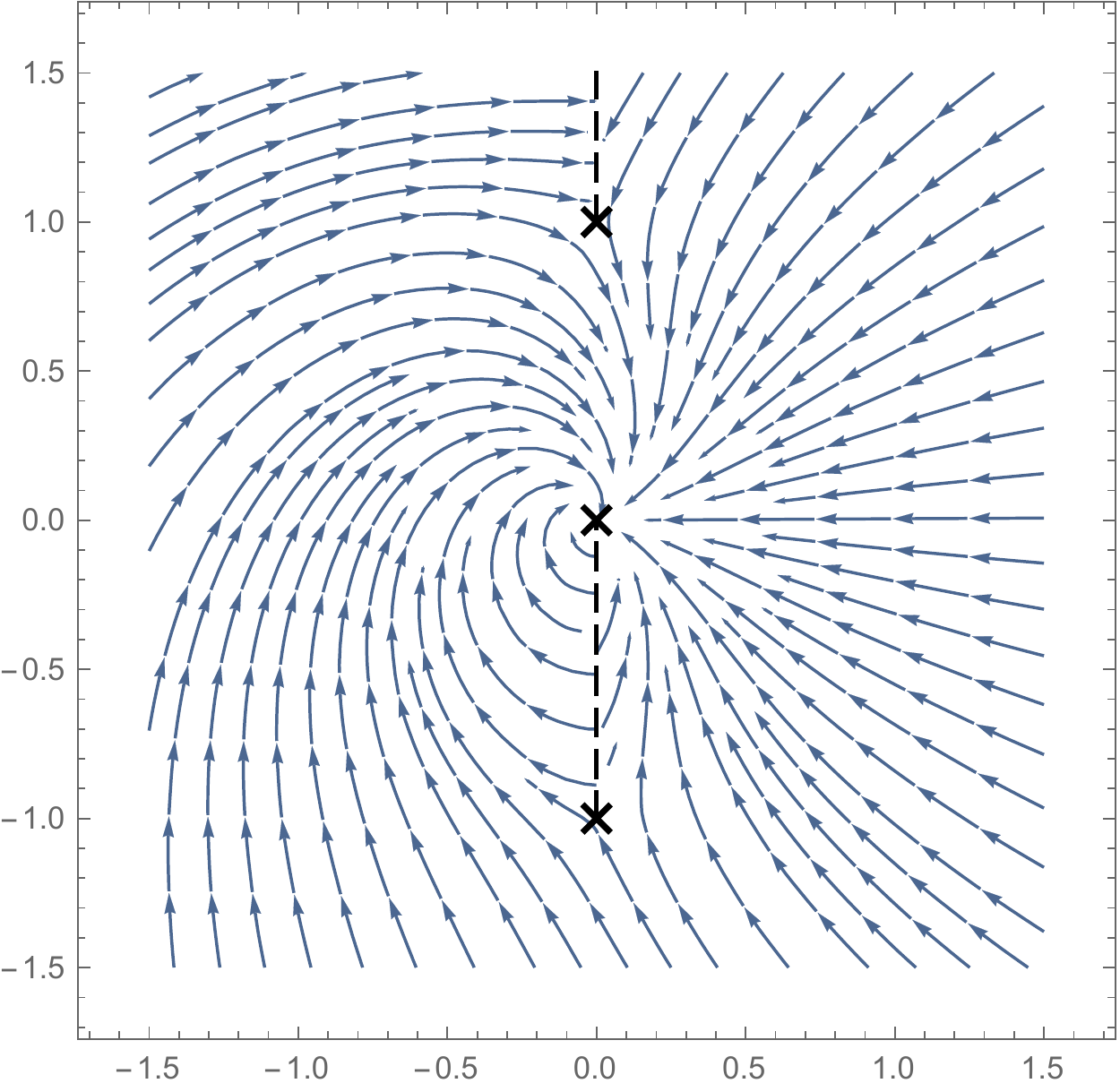}
\caption{Trajectories $z(t)$ satisfying $\arg z'(t) =-\frac{\pi}{2}+ \arg \frac{z(t)}{(z(t)+z(t)^{-1})^{1/3}}$.  \Cref{lem:walls} says that once a trajectory enters the unit disk it will always be moving closer to the origin.}
\label{fig:todavectorfield}
\end{figure}

\begin{proof}
For $z \nin \{0,i,-i\}$ write \[\th_z \coloneqq \phi + \arg\frac{z}{(z+z^{-1})^{1/N}} - \arg z.\]  Below, the branch of ${(z+z^{-1})^{1/N}}$  we mean will always be specified via analytic continuation from some $z(t)$.  In particular, $\frac{d}{dt}|z(t)|$ is negative, zero, or positive exactly when $\th_{z(t)}$ lies in $(\frac{\pi}{2},\frac{3\pi}{2})$, $\{\frac{\pi}{2},-\frac{\pi}{2}\}$, or $(-\frac{\pi}{2},\frac{\pi}{2})$, respectively.  Likewise, $\frac{d}{dt}\arg z(t)$ is negative, zero, or positive exactly when $\th_{z(t)}$ lies in $(-\pi,0)$, $\{0,\pi\}$, or $(0,\pi)$, respectively.  

Let $\CR$ denote the right unit half-disk, $\del \CR_a \coloneqq \{ ri: 0<r<1 \}$, $\del \CR_b \coloneqq \{z: \Re z > 0, |z|=1\}$, and $\del \CR_c \coloneqq \{ -ri: 0<r<1 \}$, so that
\[\del \CR = \del \CR_a \cup \del \CR_b \cup \del \CR_c \cup \{0,i,-i\}.\]  
Any branch of $\th_z$ on $\CR$ takes constant values $\th_a$, $\th_b$, and $\th_c$ along each of $\del \CR_a$, $\del \CR_b$, and $\del \CR_c$.  We have $\th_b = \th_c + \frac{\pi}{2N}$, $\th_a = \th_c + \frac{\pi}{N}$, and $\th_c <\th_z <\th_a$ for all $z$ in the interior of $\CR$.  More precisely, for fixed $0<r <1$, $\th_{re^{i\vphi}}$ goes from $\th_c$ to $\th_a$ as $\vphi$ goes from $-\frac\pi2$ to $\frac\pi2$, and for fixed $\vphi \in (-\frac\pi2,\frac\pi2)$ $\th_{re^{i\vphi}}$ goes from $\th_a + \frac{2\vphi+\pi}{2N}$ to $\th_b$ as $r$ goes from $0$ to $1$.  Analogous statements hold replacing $\CR$ by $-\CR$ throughout. 

Suppose that $z(t) \in \CR$ and $\th_{z(0)} \in (\frac\pi2,\frac{3\pi}{2})$, where $\th_{z(0)} \coloneqq \lim_{t \to 0}\th_{z(t)}$ (note that we may have $z(0) = \pm i$ even though $\th_{\pm i}$ is undefined).  Consider the branch of $\th_z$ on $\CR$ determined by $z'(t)$.  We consider several cases depending on the values of $\th_a$ and $\th_c$:

\begin{enumerate}
\item $\mathbf{\th_c \leq \pi \leq \th_a}$:  If $z(t)$ never leaves $\CR$ the claim follows since $\th_{z(t)} \in (\frac\pi2,\frac{3\pi}{2})$ for all $t$. Otherwise there is some $t_1<R$ with $z(t) \in \CR$ for $0\leq t \leq t_1$ and $z(t) \nin \CR$ for all sufficiently small $t>t_1$.  But the hypothesis on $\th_a$, $\th_c$ implies that $z'(t_1)$ cannot be directed towards the complement of $\CR$ regardless of where $z(t_1)$ lies along $\del \CR$, a contradiction.  

\item$\mathbf{\frac\pi2 \leq \th_c < \th_a < \pi}$:  Again if $z(t)$ never leaves $\CR$ the claim follows, otherwise let $t_1<R$ be as above.  We necessarily have $z(t_1) \in \del \CR_a$ since this is the only part of the boundary along which $z'(t_1)$ could be directed towards the complement of $\CR$.  But now inductively suppose suppose that for some $t_n$, $z(t_n) \in (-1)^n\del \CR_c$ and $z(t) \in (-1)^n\CR$ for all sufficiently small $t > t_n$.  Let $\th^n_z$ denote the relevant branch of $\th_z$ on $(-1)^n \CR$, and $\th^n_a$, $\th^n_c$ its values on $(-1)^n \del \CR_a$, $(-1)^n \del \CR_c$.  Suppose also that $\th^n_a = \th_c + \frac{n\pi}{N}$. If $z(t)$ never leaves $(-1)^n\CR$ for $t > t_n$ or if $\th^n_c \leq \pi \leq \th^n_a$ the claim follows.  Otherwise, there is some minimal $t_{n+1} > t_n$ satisfying the above hypotheses.  But since $\th^n_a = \th_c + \frac{n\pi}{N}$ we eventually have $\th^n_c\leq \pi \leq\th^n_a$. 

\item $\mathbf{\frac\pi2 - \frac{\pi}{N}<\th_c<\frac\pi2}$:  Let $\CR^{in} = \{z \in \CR: \th_z \in [\frac\pi2,\frac{3\pi}{2}]\}$.  There is an arc $\del \CR_{c}^{in} \subset \del \CR^{in}$ with $\th_z = \frac\pi2$ for all $z \in\del \CR_{c}^{in}$, and which intersects every circle of radius $r \in (0,1)$ around the origin at exactly one point.  Depending on whether $\th_b \in (0,\frac\pi2)$, $\th_b = \frac\pi2$, or $\th_b \in (\frac\pi2,\pi)$, one endpoint of $\del \CR_{c}^{in}$ is $0$ and the other is $i$, $1$, or $-i$, respectively.  Again, if $z(t)$ never leaves $\CR^{in}$ the claim follows, otherwise it must cross $\del \CR_a$ as it leaves and we are in the inductive situation above.
\end{enumerate}

The remaining cases $\th_{z(0)} \in (\pi,\frac{3\pi}{2})$ or $\th_{z(0)}\in (-\frac\pi2,\frac\pi2)$ can be dealt with through a simple modification of this argument.
\end{proof} 

In general any wall of $\CW_\th$ will be coincident with several other walls with distinct labels.  We refer to a maximal collection of coincident walls as a multiwall.  In $\CW_\th$ there are $2(N+1)$ multiwalls born at $i$, and up to rotation and choice of branch cut their label sets are the same as those of $\stdknet$ (likewise for $-i$).\footnote{This holds for any network where multiplication by $N$th roots of unity acts transitively on the sheets of the spectral curve, in which case the possible label sets correspond to the $2(N+1)$ possible values of $\arg(\om_{ij})$.}  Note that as $z \to \pm i$, 
\[
\frac{(z+z^{-1})^{1/N}}{z} \sim \mp 2i (z\mp i)^{1/N}.\]
  In particular, for generic $\th$, $N+1$ of the multiwalls born at $i$ lie initially inside the unit disk and $N+1$ lie initially in its complement.  It follows from \cref{lem:walls} that these multiwalls in fact lie entirely inside the unit disk or its complement, respectively.  We take those lying in the interior of the unit disk to define a splitting of $\CW_\th$ at $i$ in the sense of \cref{def:splitting}.  Following \cref{def:splittingdata} we have associated sets $\Phi_\th$ and $\De$ of positive and simple positive roots of $SL_{N}$, and a permutation $\si$ defined by 
\[
\De = \{\si_0\si_1,\si_1\si_2,\dotsc,\si_{N-2}\si_{N-1}\}.\]
  We likewise have subsets $\De_a$ and $\De_b$ of $\De$ consisting of the labels of the first and second-to-last multiwalls lying in the unit disk, taken in clockwise order around $i$.  We will also 

We define a second choice of simple positive roots
\[
\De' \coloneqq \{ \tau_a(ji): ij \in \De \},
\]
where $\tau_a$ is the product of the simple reflections across the elements of $\De_a$ (these commute, so we needn't specify their order).  We define $\Phi_\th'$, $\De'_a$, and $\De'_b$ similarly, noting that $\De'_a = \De_a$.  Equivalently, $\De'$ is the choice of simple roots associated with the splitting of $\CW_\th$ at $i$ composed of walls lying in the complement of the unit disk, relative to a branch cut isotoped from our original cut by a clockwise rotation to initially lie tangent to the unit circle.\footnote{Though from this point of view we have reversed our convention about which subset is $\De'_a$ and which is $\De'_b$.}  

For $0\leq j,k < N$, let $R_{jk}$ (resp. $L_{jk}$) denote the oriented simple closed curves on $\Si$ whose projection to $\BC^*$ is the right (resp. left) unit semicircle, and which crosses from sheet $j$ to sheet $k$ over $-i$ and from sheet $k$ to sheet $j$ over $i$.  For $1 \leq k < N$ we let 
\[
\al_k \coloneqq \si_{k-1}\si_k \in \De, \quad \al'_k \coloneqq \tau_a(\si_k\si_{k-1}) \in \De',
\]
and
\[
\ga_k \coloneqq R_{\al_k}, \quad \ga'_k \coloneqq L_{\al'_k}.
\]

Let us recall some terminology concerning BPS spectra; for a mathematical audience we will simply define the relevant notions directly in terms of a spectral curve.

\begin{defn}\label{def:bpsstring}
Let $C$ be a Riemann surface and $\Si \subset T^*C$ a branched cover of $C$.  An oriented path $p_c$ in $C$ is a BPS $ij$-string of phase $\th$ if $i$ and $j$ are sheets of $\Si$ over $p_c$ and
\begin{equation}\label{eq:bps}
\arg \<p_c'(t)|\la^{(i)}-\la^{(j)}\> = \th
\end{equation}
for any oriented parametrization of $p_c$.  Here $\la^{(i)}$ and $\la^{(j)}$ are the 1-forms on $C$ given by the sheets $i$ and $j$.
\end{defn} 

\begin{defn} A finite BPS web of phase $\th$ is a union of finitely many compact BPS strings of phase $\th$ such that an endpoint of any $ij$-string is either a branch point where sheets $i$ and $j$ meet, or a junction where three BPS strings in the web meet.  The labels of the strings meeting at a junction should be of the form $ij$, $jk$, and $ki$, and these strings should be oriented all towards or all away from the junction.
\end{defn} 

Each BPS $ij$-string $p_c$ has an associated lift to a pair of segments in $\Si$.  One is its preimage on sheet $j$ with the same orientation and the other its preimage on sheet $i$ with the reversed orientation.  Given a finite BPS web $\fweb$, the union $\fweb_\Si$ of the lifts of its BPS strings is a closed oriented path in $\Si$.  The following definition is not quite standard as it does not record multiplicities; rather, it is the minimal notion we need to talk about cluster coordinates on a toric nonabelianization chart.  

\begin{defn} The BPS spectrum of a curve $\Si \subset T^*C$ is
\[
\bps = \{ [\fweb_\Si]: \fweb_\Si \text{ is the lift of a finite BPS web}\} \subset H_1(\Si,\BZ).
\]
Each phase $\th$ determines a subset of positive charges
\[
\bps_+ = \{ [\fweb_\Si] \in \bps: \th < \arg Z(\fweb_\Si) < \th+\pi \}.
\]
Here $Z$ denotes the central charge
\[
Z: \ga \mapsto \int_\ga \la,
\]
where $\la$ is the restriction of the Liouville form on $T^*C$.  Let $\basis \subset \bps_+$ be the unique subset with the property that any element of $\bps_+$ may be written uniquely as a positive integral combination of elements of $\basis$, if such a subset exists.  The BPS quiver $Q_{\Si,\th}$ has vertices indexed by the elements of $\basis$ and $\max\{0,\<\ga,\ga'\>-\<\ga',\ga\>\}$ arrows from vertex $[\ga]$ to vertex $[\ga']$.
\end{defn}

The strong-coupling BPS spectrum of pure $\CN=2$ $SU(N)$ Yang-Mills was computed in \cite{Lerche2000} and later in \cite{Alim2011} using different methods.  Below we show how \cref{lem:walls} lets us exactly calculate the finite BPS webs of the Toda spectral curve $\Si$, recomputing the strong-coupling spectrum directly as defined in terms of BPS strings (mathematically, these previous computations refer to distinct ways of defining the spectrum).  A Mathematica computation of the finite webs in the $N=3$ case was given in \cite{Gaiotto}.  Note that below we consider the closed curve $\ol{\Si}$ in place of the open curve $\Si$.

\begin{prop}\label{prop:todabps}
With respect to a generic phase $\th$, the positive part of the BPS spectrum of $\Si$ is 
\[
\BB_+ = \{[R_{ij}]: ij \in \Phi_\th\} \cup \{[L_{ij}]: ij \in \Phi_\th'\}
\]
and its positive integral basis is 
\[
\basis = \{[\ga_i], [\ga'_i]: 1 \leq i < N\}.
\]
The BPS quiver is $Q_{N-1}$, where under our labeling $\ga_j$ corresponds to vertex $2j$ (resp. $2j-1$) if $\al_j \in \De_b$ (resp. $\De_a$), and $\ga'_j$ corresponds to vertex $2j$ (resp. $2j-1$) if $\al'_j \in \De'_a$ (resp. $\De'_b$).  
\end{prop}

\begin{proof}  
Suppose $\fweb$ is a finite BPS web that intersects the interior of the unit disk nontrivially.  Since $\fweb$ is a finite union of closed segments, a minimum value of $|z|$ is attained at some $z_{min} \in \fweb$.  It follows from \cref{lem:walls} that if $p_c$ is a BPS string in $\fweb$ with $z_{min}\in p_c$, then $z_{min}$ must lie at an endpoint of $p_c$.  In particular, $z_{min}$ has to lie at a junction of three BPS strings.  But the labels of these are of the form $ij$, $jk$, and $ki$, and it follows from \cref{eq:bps} that if $z_{min}$ minimizes $|z|$ on two of these strings it cannot minimize it on the third, a contradiction.  Similarly there are no finite webs that intersect the complement of the unit disk nontrivially.

On the other hand, as $\vphi$ varies from $\th$ to $\th+\pi$, the walls of $\CW_\vphi$ born at $i$ rotate around $i$ and degenerate to finite webs covering the right and left unit semicircles at certain critical values of $\vphi$.  The canonical lifts of these finite webs are exactly the paths $\{R_{ij}: ij \in \Phi_\th\} \cup \{L_{ij}: ij \in \Phi_\th'\}$, and any finite web which includes a BPS string lying on the unit circle arises in this way.  

The cycles $\ga_i$, $\ga'_i$ comprise the simple roots $\De \subset \Phi_\th$, $\De' \subset \Phi_\th'$, hence make up the basis $\basis$.  Noting that they can be deformed to intersect only at the branch points, it is an elementary computation to derive the quiver $Q_{N-1}$ from their intersection numbers. 
\end{proof}

Let $\CM_{SL_N}(C)$ denote the wild $SL_N$-character variety of $C=\BC^*$ related to the periodic Toda phase space by the wild nonabelian Hodge correspondence.  We will not require a careful definition of this space or of the relevant Stokes data at $0$, $\infty$.  Indeed, there are unresolved technical issues in extending the nonabelian Hodge correspondence to the present setting, as the treatment in \cite{Biquard2000} makes certain semisimplicity assumptions not satisfied by Toda systems.  However, it is expected that this assumption can be lifted, see for example \cite{Witten2007}.  More to the point, these analytic issues are orthogonal to our immediate goals.  For our purposes it is enough that we can use the spectral network $\CW_\th$ to obtain an unambiguous definition of the pullback from $\CM_{SL_N}(C)$ to $\CM_{GL_1}(\ol{\Si})$  of the trace of the holonomy around a closed path in $C$.

Recall that the spectral network $\CW_\th$ defines a nonabelianization map $\Psi_{\CW_\th}: \CM_{GL_1}^{tw}(\ol{\Si}) \to \CM_{SL_N}^{tw}(C)$ between spaces of twisted local systems.  Passing to untwisted local systems requires a choice of spin structure.  Recall that a spin structure on $\ol{\Si}$ may be given as a quadratic refinement 
\[
\qr: H_1(\ol{\Si},\BZ) \to \BZ/2\BZ,\quad \qr(\ga+\ga') = (-1)^{\<\ga,\ga'\>}\qr(\ga)\qr(\ga').
\]
  For each $\ga \in H_1(\ol{\Si},\BZ)$, let $y_\ga \in \BC[\CM_{GL_1}(\ol{\Si})]$ be the holonomy around $\ga$, and $\wh{y}_\ga \in \BC[\CM_{GL_1}^{tw}(\ol{\Si})]$ the corresponding twisted holonomy (that is, the holonomy around the canonical homology lift of $\ga$ \cite{Johnson1980}).  A quadratic refinement $\qr$ defines an isomorphism 
\[
\io_\qr: \CM_{GL_1}(\ol{\Si}) \to \CM_{GL_1}^{tw}(\ol{\Si}),\quad \io_\qr^*: \wh{y}_\ga \mapsto (-1)^{\qr(\ga)}y_\ga.
\]  

From now on let $\si$ denote the following quadratic refinement $\qr$ on $\ol{\Si}$.  Let
\[
B_i \coloneqq \begin{cases} \ga_i & \al_i \in \De_a \\ \ga'_i & \al_i \in \De_b.\end{cases}  
\]
Let $D_i$ be the homology class of the simple closed curve in $\ol{\Si}$ which covers the clockwise-oriented unit circle, lies on sheet $\si_i$ above the right unit semicircle, and lies on sheet $\tau_a\si_i$ above the left unit semicircle.  Then if $A_j \coloneqq \sum_{k \geq j}D_k$ for $1\leq j \leq N$, the reader may check that $\<A_i,B_j\> = \de_{ij}$.  We then define $\qr$ by 
\[
\qr(A_i) = 1, \quad \qr(B_i) = -1.
\]
Note that
\[
\ga_i = \begin{cases} B_i & \al_i \in \De_a \\ A_{i-1} - 2A_i + A_{i+1} - B_i & \al_i \in \De_b \end{cases},\quad \ga'_i = \begin{cases} B_i & \al_i \in \De_b \\ A_{i-1} - 2A_i + A_{i+1} - B_i & \al_i \in \De_a \end{cases},
\]
where we let $A_0 = A_N = 0$.  In particular, $\qr(\ga_i) = \qr(\ga'_i) = -1$ for all $i$, which will be the main property we care about.  

Let $y_1,\dotsc,y_{2(N-1)} \in \BC[\CM_{GL_1}(\ol{\Si})]$ denote the holonomies around the elements of $\basis$, indexed according to the identification of $\basis$ with the vertices of $Q_{N-1}$ described in \cref{prop:todabps}.  Let $\wp$ denote the clockwise-oriented unit circle, $V_k = \bigwedge^k\BC^N$, and $\tr_{V_k}\Hol_{\ti\wp} \in \BC[\CM^{tw}_{SL_n}(C)]$ the trace in $V_k$ of the holonomy around the canonical lift $\ti{\wp}$ of $\wp$ to the unit tangent bundle.  The network $\CW_\th$ lets us pull back $\tr_{V_k}\Hol_{\ti\wp}$ along $\Psi_{\CW_\th}$ to $\CM^{tw}_{GL_1}(\ol{\Si})$, and the quadratic refinement $\qr$ lets us pull it further back along $\io_\qr$ to $\CM_{GL_1}(\ol{\Si})$.  

The sublattice $\Ga \subset H_1(\ol{\Si},\BZ)$ generated by $\basis$ has index $N$, and the maps
\[
\Ga \hookrightarrow H_1(\ol{\Si},\BZ) \hookrightarrow \Ga^*
\]
induce maps of tori
\[
\CX_{Q_{N-1}} \twoheadleftarrow \CM_{GL_1}(\ol{\Si}) \twoheadleftarrow \CA_{Q_{N-1}}.
\]
Here $\CX_{Q_{N-1}} \cong \Hom(\Ga,\BC^*)$ and $\CA_{Q_{N-1}} \cong \Hom(\Ga^*,\BC^*)$ as in \cref{sec:clusterapp} and the composition of these maps is the usual map $y_j \mapsto \prod_i x_i^{Q_{ji}}$.  With this in mind we conflate $\io_\qr^*\Psi_{\CW_\th}^*\tr_{V_k}\Hol_{\ti\wp}$ with its expression in coordinates on $\CA_{Q_{N-1}}$ in order to compare it with the relativistic Toda Hamiltonian $H_k$.
 
\begin{thm}\label{thm:holotonian}
For generic $\th$ and $1 \leq k \leq N$, we have $H_k = \io_\qr^*\Psi_{\CW_\th}^*\tr_{V_k}\Hol_{\ti\wp}.$
\end{thm}

\begin{proof}
Following \cref{cor:simpleres}, we can resolve $\CW_\th$ by an equivalent simple spectral network $\CW'_\th$ subordinate to a cover $\Si'$ which deforms $\Si$.  Recall from \cref{prop:weightedsum} that the expansion of $H_k$ in cluster coordinates is the weighted sum of nonintersecting $k$-tuples of closed directed paths in the directed annular graph $\ol{\CN}_{\mb{i}}$.  Here we define another graph $\ol{\CN}_\CW$, equivalent to $\ol{\CN}_{\mb{i}}$ in the sense described in \cref{sec:paths}.  This graph will be embedded into $\Si'$ as the union of the paths appearing in the expansion $F(\wp,\CW'_\th)$ used to compute $\Psi_{\CW_\th}^*\tr_{V_k}\Hol_{\ti\wp}$.  This embedding induces an isomorphism on first homology groups that identifies  $\basis$ with the set of face cycles of $\ol{\CN}_\CW$ and $\ol{\CN}_{\mb{i}}$.  The result then follows by establishing a suitable bijection between the terms appearing in $F(\wp,\CW'_\th)$ and the maximal directed paths in $\ol{\CN}_\CW$ and $\ol{\CN}_{\mb{i}}$.  

Analogously with $\CN_{\mb{i}}$, we define a graph $\CN_{\CW}$ via the following embedding into $[0,1]^2$:
\begin{align*}
\CN_{\CW} =& \left( \bigcup_{k=0}^{N-1} [0,1]\times\{\frac{k+1}{N+1}\} \right)
\cup \left( \bigcup_{\al_i \in \De_b}\{\frac15,\frac45\} \times [N+1-i,N-i] \right)\\
&\quad
\cup \left( \bigcup_{\al_i \in \De_a}\{\frac25,\frac35\} \times [N+1-i,N-i] \right).
\end{align*}
The vertical edges of the form $\{\frac15\} \times [N+1-i,N-i]$ and $\{\frac35\} \times [N+1-i,N-i]$ are directed upward, those of the form $\{\frac25\} \times [N+1-i,N-i]$ and $\{\frac45\} \times [N+1-i,N-i]$ are directed downward, and the horizontal edges are directed leftward.  We let $\ol{\CN}_\CW$ denote the closed directed graph which is the image of $\CN_\CW$ in $S^1\times [0,1]$ after identifying the left and right edges of $[0,1]^2$.  The graph $\ol{\CN}_{\CW}$ is equivalent to $\ol{\CN}_{\mb{i}}$ in the sense described in \cref{sec:paths}; in particular, there is a canonical isomorphism $H_1(\ol{\CN}_{\CW},\BZ) \cong H_1(\ol{\CN}_{\mb{i}},\BZ)$ and a bijection between their closed directed paths.  

Consider the splitting of $\CW_\th$ at $-i$ consisting of walls lying in the complement of the unit disk.  Along with our chosen splitting at $i$ of walls lying inside the unit disk, this determines a simple network $\CW'_\th$ subordinate to a cover $\Si'$.  The path $\wp$ intersects exactly $2(N-1)$ walls of $\CW'_\th$ in the following order.  Let
\[
\De^{op} \coloneqq \{ ij: ji \in \De\},
\]
 defining $\De_a^{op}$, $\De_b^{op}$ analogously.  Then as $\wp$ passes through a neighborhood of $i$ (resp. $-i$) it crosses a series of walls labeled by $\De_b$ (resp. $\De_a$) then another series labeled by $\De_a^{op}$ (resp. $\De_b$).  We may choose branch cuts so that after crossing the $N-1$ walls near $i$, $\wp$ crosses a series of branch cuts labeled by the simple reflections around the elements of $\De_a$, and then crosses an identically labeled series of branch cuts before crossing any of the walls near $-i$.

We now take $\wp$ to have its basepoint at 1, and consider the expansion
\[F(\wp,\CW'_\th) = \sum_{\ba} \fro'(\ti \wp,\CW'_\th,\ba) \bX_\ba.
\]
We claim that the union of the open paths $\ba$ for which $\fro'(\ti \wp,\CW'_\th,\ba) \neq 0$ defines an embedding of $\ol{\CN}_\CW$ into the unit tangent bundle of $\Si'$.  This union is composed of the preimages of the canonical lifts of $\wp$, which we identify with the horizontal paths in $\ol{\CN}_\CW$, and segments with endpoints on these preimages which lie along canonical lifts of the walls crossed by $\wp$, which we identify with the vertical paths in $\ol{\CN}_\CW$.  Specifically, the horizontal edge $[0,1]\times\{\frac{k+1}{N+1}\}$ is identified with the lift of the path $D_k$ introduced in the definition of $\qr$ above.  The vertical walls of the form $\{\frac15\} \times [N+1-i,N-i]$ (resp. $\{\frac15\} \times [N+1-i,N-i]$) are identified with segments of the canonical lifts of the walls labeled by $\De_b^{op}$ (resp. $\De_a$) crossed by $\wp$ near $-i$ (plus small segments along the unit tangent circles above the intersections of $\wp$ and these walls as in \cref{eq:F-one-cross}).  Likewise, the vertical walls of the form $\{\frac35\} \times [N+1-i,N-i]$ (resp. $\{\frac45\} \times [N+1-i,N-i]$) are identified with the segments of the canonical lifts of the walls labeled by $\De_a^{op}$ (resp. $\De_b$) crossed by $\wp$ near $i$. 

After projecting to $\Si'$, this embedding induces an isomorphism of $H_1(\ol{\CN}_\CW,\BZ)$ with $H_1(\Si',\BZ)$.  Letting $\ol{\Si}'$ denote the smooth closure of $\Si'$, the projection to $H_1(\ol{\Si'},\BZ)$ further identifies the face cycles of $\ol{\CN}_\CW$ (taken with counterclockwise orientation) with $\basis$.  The reader may check that this identification is compatible with the identification of each set of cycles with the vertices of $Q_{N-1}$.   

The open paths $\ba$ for which $\fro'(\ti \wp,\CW'_\th,\ba) \neq 0$ are exactly the images of the maximal directed paths in $\CN_\CW$ (that is, directed paths starting and ending on boundary vertices of $\CN_\CW$).  In particular, we have
\[
\Psi_\CW^* \tr_{V_k}\Hol_\wp = \sum_{\{a_i\}_{i=1}^k}\prod_{i=1}^k\wh{y}_{a_i},
\]
where the sum is over nonintersecting $k$-tuples of closed paths with $\fro'(\ti \wp,\CW'_\th,a_i) \neq 0$, hence over nonintersecting $k$-tuples of closed directed paths in $\ol{\CN}_\CW$.  

To finish the proof we must apply $\io_\qr^*$ to this expression and expand it in the coordinates $y_1,\dotsc,y_{2(N-1)}$.  We claim that $\qr(a) = 1$ when $a$ is the image of a closed directed path in $\ol{\CN}_\CW$.  Recall that $\qr$ was defined with respect to a symplectic basis on which $\qr(A_i) = -1$ and $\qr(B_i) = 1$.  All images of closed directed paths in $\ol{\CN}_\CW$ have homology classes of the form $D_i$ or $D_i + A_{i+1}$.  Each $D_i$ can be written as a sum of $B$-cycles, so $\qr(D_i) = 1$.  On the other hand, $D_i$ and $A_{i+1}$ intersect once, so we also have \[\qr(D_i + A_{i+1}) = - \qr(D_i)\qr(A_{i+1}) = 1.\]  

Recall from \cref{sec:paths} that weights of closed directed paths in $\ol{\CN}_{\mb{i}}$ were determined by normalizing the weight of the lowest path to be
\begin{align*}
\wht(p_{\min})=\prod_{i=1}^n(y_{2i-1}y_{2i})^{-\frac{i}{n+1}}.
\end{align*}
The weights of other paths were then determined by the homology class of their difference from the lowest path.  Given our identification of face cycles in $\ol{\CN}_{\mb{i}}$ with $\basis$, we are done once we show that the above expression coincides with the pullback from $\CM_{GL_1}(\ol{\Si})$ of the holonomy $y_{D_N}$.  But this follows from the fact that 
\begin{align*}
ND_N + \sum_{k=1}^N k(\ga_k + \ga'_k) = 0
\end{align*}
in $H_1(\ol{\Si'},\BZ)$, since the cycle $D_0 + \cdots + D_N$ becomes trivial in the closed surface $\ol{\Si'}$ and 
\[
D_j = D_N + \sum_{i=j+1}^{N}(\ga_i + \ga'_i)
\]
 for all $0 \leq j \leq N$.
\end{proof}

\begin{figure}
\includegraphics[scale=.7]{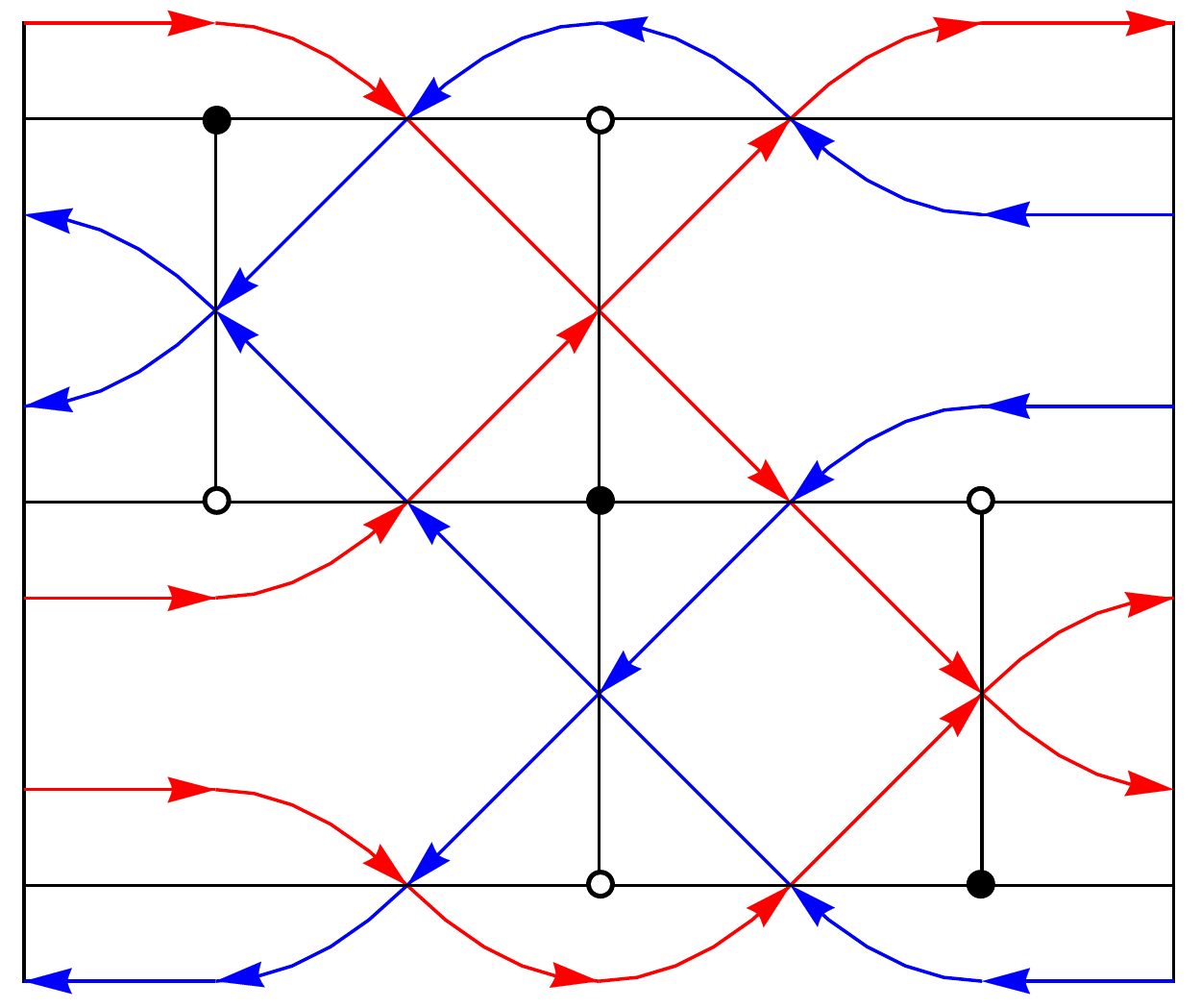}
\caption{The bipartite form of $\ol{\CN}_{\mb{i}}$ and its two zig-zag paths in the $SL_3$ case.}
\label{fig:bipartite}
\end{figure}

\begin{remark}
It is illuminating to view the proof of \cref{thm:holotonian} in the context of the conjugate surface construction of \cite{Goncharov2011}.  An embedding of a bipartite graph $G$ into a surface $S$ defines a ribbon structure on $G$, and the conjugate surface $\wh{S}$ is the ribbon graph obtained by gluing in a half-twist in the middle of every edge of $G$.  The boundary components of $\wh{S}$ are identified with the ``zig-zag paths'' of $G$, as pictured in \cref{fig:bipartite}.  On the other hand, there is a standard way to associate a bipartite surface graph to a directed surface graph \cite{Postnikov2006}.  The above proof implictly identifies the conjugate surface of $\ol{\CN}_{\mb{i}}$ with the periodic Toda spectral curves as topological surfaces.  In particular, the fact that there are exactly two zig-zag paths of $\ol{\CN}_{\mb{i}}$, hence two boundary components of its conjugate surface, corresponds to the fact that the spectral curve has cyclic monodromy at $0$ and $\infty$.  
\end{remark}

\appendix
\section{Background on Cluster Algebras}\label{sec:clusterapp}

In this appendix we fix our conventions regarding cluster algebras and cluster varieties, reviewing the minimal part of the theory we need \cite{Fomin2006}.  We denote cluster variables by $x_i$ and their dual coordinates by $y_i$, as this follows the conventions in the literature on cluster characters most closely.  However, it is also convenient to have notation for the tori on which these provide coordinates, hence we also use the language of $\CA$- and $\CX$-tori at times \cite{Fock2003}.

Let $Q$ be a quiver with vertices $Q_0 = \{1,\dotsc,n\}$ and no oriented 2-cycles.  To $Q$ we associate a pair of tori
\[
\CA_Q = \Spec \BC[x_1^{\pm 1},\dotsc,x_n^{\pm 1}], \quad \CX_Q= \Spec \BC[y_1^{\pm 1},\dotsc,y_n^{\pm 1}]
\]
and a map
\[
p_Q: \CA_Q \to \CX_Q,\quad p_Q^*: y_i \mapsto \prod_{j=1}^n x_j^{Q_{ij}}.
\]
Here
\[
Q_{ij} = |\{\text{arrows }i \to j\text{ in }Q\}| - |\{\text{arrows }j \to i\text{ in }Q\}|
\]
is an entry of the signed adjacency matrix of $Q$.  The data of $Q$ is equivalent to the data of a lattice $\Ga$ with a basis $\{e_i\}$ and a skew-symmetric form: to such a lattice we associate a quiver by setting
\[
Q_{ij} = \langle e_i, e_j \rangle.
\]
From this point of view
\[
\CA_Q = \Hom(\Ga^*,\BC^*),\quad \CX_Q = \Hom(\Ga,\BC^*),
\]
their coordinates coming from the basis $\{e_i\}$ and its dual.  The skew-symmetric form induces a map $\Ga \to \Ga^*$ by $e_i \mapsto \langle e_i, -\rangle$, which gives rise to the map $p_Q$.

Given $1 \leq k \leq n$, the mutation $\mu_k(Q)$ of $Q$ at $k$ is the quiver with $\mu_k(Q)_0 = Q_0$ and
\begin{align*}
\mu_k(Q)_{ij} = \begin{cases}
-Q_{ij} & i = k \text{ or } j=k \\
Q_{ij} + \frac12(|Q_{ik}|Q_{kj} + Q_{ik}|Q_{kj}|) & i,j \neq k. \xqedhere{3.98cm}{\qedhere}
\end{cases}
\end{align*}

To the mutation $\mu_k$ we associate a pair of birational maps, called cluster transformations and also denoted by $\mu_k$.  These satisfy
\[
\begin{tikzcd}
\CA_{Q} \arrow[dashed]{r}{\mu_k} \arrow{d}{p_{Q}} & \CA_{Q'} \arrow{d}{p_{Q'}} \\
\CX_{Q} \arrow[dashed]{r}{\mu_k} & \CX_{Q'}
\end{tikzcd},
\]
where $Q'=\mu_k(Q)$, and are defined explicitly by
\begin{gather*}
\mu_k^*(x'_i) = \begin{cases}
x_i & i \neq k \\
\displaystyle x_k^{-1}\biggl(\prod_{Q_{kj}>0}x_j^{Q_{kj}} + \prod_{Q_{kj}<0}x_j^{-Q_{kj}}\biggr) & i = k
\end{cases}
\end{gather*}
and
\begin{gather*}
\mu_k^*(y'_i) = \begin{cases}
y_iy_k^{[Q_{ik}]_+}(1+y_k)^{-Q_{ik}} & i \neq k  \\
y_k^{-1} & i = k,
\end{cases}
\end{gather*}
where $[Q_{ik}]_+\coloneqq\mathrm{max}(0,Q_{ik})$.  The map $\mu_k: \CX_Q \to \CX_{Q'}$ intertwines the Poisson bracket
\[
\{y_i,y_j\} = Q_{ij}y_iy_j
\]
and its counterpart on $\CX_{Q'}$.  

The $\CA$- and $\CX$-spaces $\CA_{|Q|}$ and $\CX_{|Q|}$ are the schemes obtained from gluing together along sequences of cluster transformations all such tori obtained iterated mutations of $Q$.  The upper cluster algebra $\ol{A}(Q)$ is the algebra of regular functions on $\CA_{|Q|}$, or equivalently
\[
\ol{A}(Q) := \BC[\CA_{|Q|}] = \bigcap_{Q' \sim Q} \BC[\CA_{Q'}] \subset \BC(\CA_{|Q|}).
\]
The cluster algebra $A(Q)$ is the subalgebra of the function field $\BC(\CA_{|Q|})$ generated by the collection of all cluster variables of seeds mutation equivalent to $Q$; the fact that $A(Q) \subset \ol{A}(Q)$ is referred to as the Laurent phenomenon.  Since cluster transformations intertwine the indicated Poisson structures on each $\CX$-torus, these assemble into a global Poisson structure on $\CX_{|Q|}$.

\section{Cluster Characters and Quivers with Potential}

In this appendix we recall the Jacobian algebra of a quiver with potential \cite{Derksen2008} and the cluster character of a module over a Jacobian algebra \cite{Palu2008}.  

Given a quiver $Q$ with oriented cycles, $\wh{\BC Q}$ is the completion of the path algebra $\BC Q$ with respect to the ideal generated by paths of length greater than zero.  A potential $W$ is an element of $Pot(\BC Q)$, the closure in $\wh{\BC Q}$ of the ideal generated by all nontrivial cyclic paths in $\BC Q$.  Given an arrow $a$ of $Q$, the cyclic derivative $\del_a:  Pot(\BC Q) \to \wh{\BC Q}$ is the unique continuous linear map such that
\[
\del_a(c) = \sum_{c = paq} qp
\]
for any cycle $c$, where the sum is taken over all decompositions of $c$ with $p,q$ being possibly length zero paths.  We call a pair $(Q,W)$ a quiver with potential, and always assume $Q$ has no oriented 2-cycles.  The Jacobian algebra $J(Q,W)$ is the quotient of $\wh{\BC Q}$ by the closure of the ideal generated by all cyclic derivatives of $W$.  We say $(Q,W)$ is Jacobi finite if $J(Q,W)$ is finite-dimensional, and from now on assume this is the case.

We write $J(Q,W)\textrm{-mod}$ for the category of finite-dimensional left $J(Q,W)$-modules; equivalently this is the category of finite-dimensional representations of $Q$ satisfying the relations imposed by the cyclic derivatives of $W$.  Given a labeling of $Q_0$ by $\{1,\dots,n\}$,  we write $S_j$ for the simple $J(Q,W)$-module supported at the $j$th vertex of $Q$, $P_j$ for its projective cover, and $I_j$ for its injective envelope.  Below we freely conflate the dual cluster variable $y_i$ with its pullback $p_Q^*y_i = \prod_{j = 1}^n x_j^{Q_{ij}}$ to $\BC[x_1^{\pm1},\dots,x_n^{\pm1}]$.

\begin{defn}
Let
\[
0 \to M \to \bigoplus_{j=1}^n I_j^{\oplus b_j} \to \bigoplus_{j=1}^n I_j^{\oplus a_j}
\]
be a minimal injective copresentation of a $J(Q,W)$-module $M$.  Then the coindex of $M$ is defined by
\[ x^{-\coind M} = \prod_{j=1}^n x_j^{a_j-b_j}. \]
\end{defn}

Given a dimension vector $e \in \BZ^n$ and a $J(Q,W)$-module $M$, the quiver Grassmannian $\Gr_e M$ is the variety of $e$-dimensional subrepresentations of $M$.  It is a projective subvariety of a linear Grassmannian of $M$, and below we write $\chi(\Gr_e M)$ for its topological Euler characteristic.  If $e_1,\dotsc,e_n \in \BZ^n$ are the dimension vectors of the the simple modules and $e = \sum a_j e_j$, we write
\[
y^e = \prod_{j=1}^n y_j^{a_j}.
\]

\begin{defn}\label{def:CC}
The cluster character $CC(M)$ of a $J(Q,W)$-module $M$ is the Laurent polynomial
\[
CC(M) = x^{- \coind M}\sum_{e \in \BZ^n} \chi(\Gr_e M) y^{e} \in \BC[x_1^{\pm1},\dots,x_n^{\pm1}].
\]
\end{defn}

This elementary definition will suffice for our purposes but deviates from \cite{Palu2008} as follows.  To obtain a seed-independent picture one may consider the cluster category $\CC$, constructed for Jacobi-finite quivers with potential in \cite{Amiot2008}.  The cluster category is triangulated, 2-Calabi-Yau, and endowed with a canonical cluster-tilting object $T$.  We have $\End_\CC(T)\cong J(Q,W)$, and the functor $\Hom_\CC(T,-)$ induces an equivalence 
\[ J(Q,W)\textrm{-mod} \cong \CC/\< \Si T \>.\] Here $\Si$ is the suspension functor of $\CC$ and $\< \Si T \>$ the ideal of all morphisms factoring through the additive subcategory generated by $\Si T$.  The cluster character can then be defined for arbitrary objects in $\CC$, and the cluster character of an indecomposable object $X$ that is not a summand of $\Si T$ agrees with \cref{def:CC} applied to $\Hom(T,X)$.

The notion of a cluster character originates in \cite{Caldero2004} for Dynkin quivers, and is treated in increasing generality in \cite{Caldero2006,Palu2008,Plamondon2011}.  It is motivated by the following property:

\begin{thm}
For a Jacobi-finite quiver with nondegenerate potential, the cluster character defines a bijection between reachable rigid indecomposable $J(Q,W)$-modules and non-initial cluster variables of $A(Q)$.
\end{thm}

\section{Background on Spectral Networks}\label{sec:snetapp}

In this appendix we recall the notation and basic notions of spectral networks, referring the reader to \cite{Gaiotto} for a comprehensive exposition. As it is needed in the body of the paper, we explicitly describe how the formalism must be extended to remove the assumption of simple ramification.  We will refer to spectral networks subordinate to branched covers with simple ramification as simple spectral networks.

Let $C$ be an orientable real surface, possibly with boundary, together with a nonempty set of points $\fs_n$.  We refer to the $\fs_n$ as singular points, and to those not on the boundary of $C$ as punctures.  Let $\Si \to C$ be an $N$-sheeted branched cover such that all singular points lie in the complement $C'$ of the branch locus.  A spectral network subordinate to the covering $\Si$ is a collection
\[
\CW = ( o(\mathfrak{s}_n), \{z_\mu\}, \{p_c\} )
\]
of the following data:

\begin{description}
\item[D1] A partially ordered subset $o(\fs_n)$ of sheets of $\Si$ in a neighborhood of each singular point $\fs_n$.  Each $o(\mathfrak{s}_n)$ must contain
at least two elements, and if $\fs_n$ is a puncture $o(\mathfrak{s}_n)$ must contain
$N$ elements.
\item[D2] A locally finite subset $\{z_\mu\} \subset C'$, referred to as joints.
\item[D3] A countable collection $\{p_c\}$ of closed segments (images of embeddings of $[0,1]$) in $C$, referred to as walls.  For each orientation $o$ of $p_c$, the segment carries a label consisting of an ordered pair of distinct sheets of $\Si$ over $p_c$.  Reversing the orientation of $p_c$ should reverse the order of the sheets, so $p_c$ has two labels which could be written as $(o, ij)$ and $(-o, ji)$.
\end{description}

The data must satisfy the following conditions:

\begin{description}

\item[C1] Two segments $p_c$ are not allowed to intersect away form their endpoints unless no sheet of $\Si$ appears in the labels of both segments.  The endpoints of the $p_c$ are at joints, branch points, or singular points.  Any compact subset of $C'$
intersects only finitely many of the $p_c$.

\item[C2] If the monodromy around a branch point $\fb$ is a single $K$-cycle, there is a neighborhood of $\fb$ in which $\CW$ is equivalent to $\CW_K$ in a neighborhood of 0 (see \cref{def:stdknet}).  That is, there should be a 1-parameter family of equivalent networks $\CW_t$ such that $\CW_0 = \CW$ and the restriction of $\CW_1$ to some neighborhood of $\fb$ is diffeomorphic with the restriction of $\CW_K$ to a neighborhood of $0$.  In a neighborhood of a general branch point $\fb$ of $\CW$ should in the same sense be locally equivalent to a superposition of $\CW_K$'s according to the cycle decomposition of the monodromy around $\fb$.  

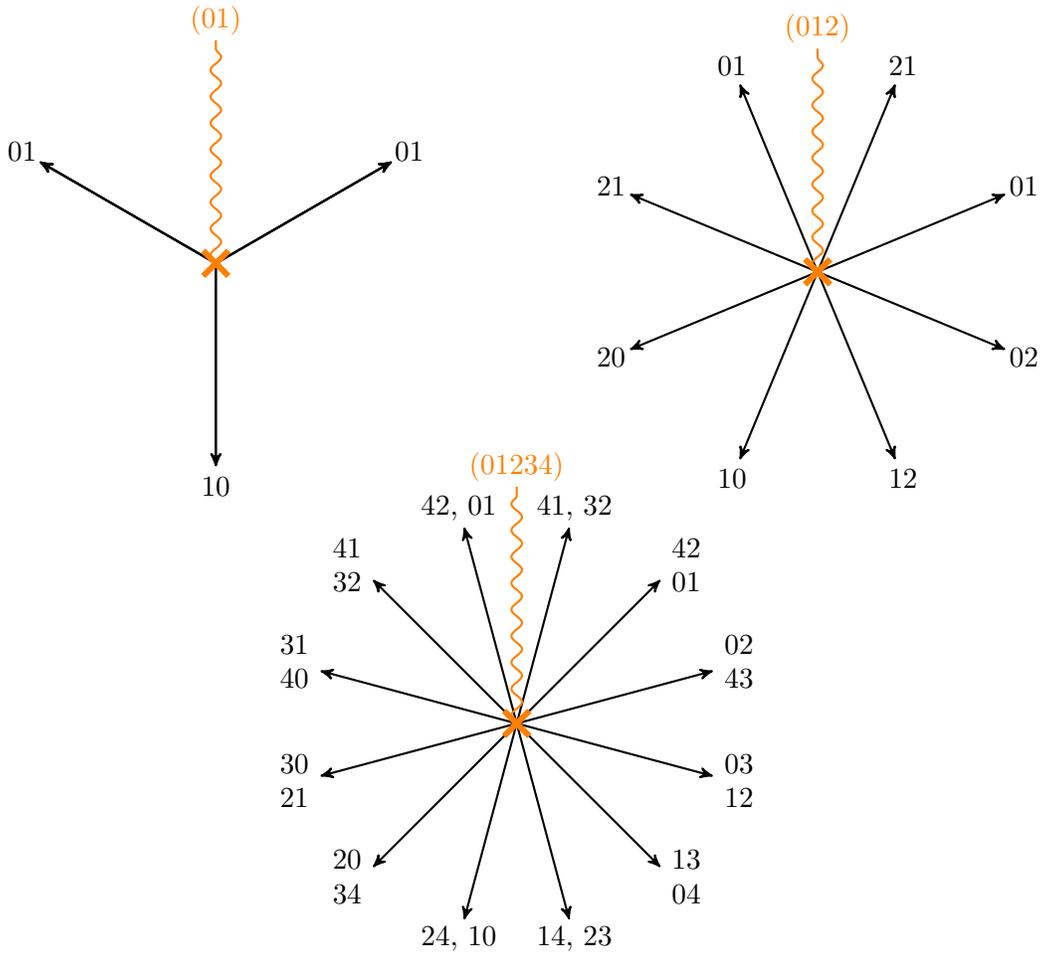
\begin{figure}
\begin{tikzpicture}[decoration={snake,amplitude=2}]
\newcommand*{\rad}{3};
\newcommand*{\scl}{.9};
\newcommand*{\hpos}{5.5};
\node (n3) [matrix,cells={scale=\scl}] at (0,0) {
\foreach \n in {0,1,...,7} {\draw [-stealth',thick] (0,0) -- (30+\n*120:\rad);};
\foreach \n/\l in {0/01,1/01,2/10} {\node at (30+\n*120:\rad+.3) {\l};};
\draw [decorate,color=orange,thick] (0,0) -- (90:\rad+.3);
\node [color=orange] at (90:\rad+.6) {(01)};
\node [cross out,draw=orange,line width=.8mm,rounded corners] at (0,0) {};
\\
};
\node (n3) [matrix,cells={scale=\scl}] at (8,0) {
\foreach \n in {0,1,...,7} {\draw [-stealth',thick] (0,0) -- (22.5+\n*45:\rad);};
\foreach \n/\l in {0/01,1/21,2/01,3/21,4/20,5/10,6/12,7/02} {\node at (22.5+\n*45:\rad+.3) {\l};};
\draw [decorate,color=orange,thick] (0,0) -- (90:\rad+.3);
\node [color=orange] at (90:\rad+.6) {(012)};
\node [cross out,draw=orange,line width=.8mm] at (0,0) {};
\\
};
\node (n5) [matrix,cells={scale=\scl}] at (4,-6) {
\foreach \n in {0,1,...,11} {\draw [-stealth',thick] (0,0) -- (15+\n*30:\rad);};

\foreach \n/\l/\m in {0/02/43,5/31/40,6/30/21,11/03/12} {\node at (15+\n*30:\rad+.4) {$\begin{matrix}\text{\l}\\ \text{\m}\end{matrix}$};};

\foreach \n/\l/\m in {1/42/01} {\node at ($(15+\n*30:\rad+.4)+(.1,-.1)$) {$\begin{matrix}\text{\l}\\ \text{\m}\end{matrix}$};};
\foreach \n/\l/\m in {4/41/32} {\node at ($(15+\n*30:\rad+.4)+(-.1,-.1)$) {$\begin{matrix}\text{\l}\\ \text{\m}\end{matrix}$};};
\foreach \n/\l/\m in {7/20/34} {\node at ($(15+\n*30:\rad+.4)+(-.1,.1)$) {$\begin{matrix}\text{\l}\\ \text{\m}\end{matrix}$};};
\foreach \n/\l/\m in {10/13/04} {\node at ($(15+\n*30:\rad+.4)+(.1,.1)$) {$\begin{matrix}\text{\l}\\ \text{\m}\end{matrix}$};};

\foreach \n/\l/\m in {3/42/01,2/41/32,8/24/10,9/14/23} {\node at (15+\n*30:\rad+.3) {\l, \m};};

\draw [decorate,color=orange,thick] (0,0) -- (90:\rad+.5);
\node [color=orange] at (90:\rad+.8) {(01234)};
\node [cross out,draw=orange,line width=.8mm] at (0,0) {};
\\
};
\end{tikzpicture}
\caption{The spectral networks $\CW_2$, $\CW_3$, and $\CW_5$ which serve as local models for a spectral network at a branch point with $2$-cyclic, $3$-cyclic, and $5$-cyclic monodromy, respectively.  The walls of $\CW_5$ come in pairs with distinct labels but whose locations coincide.}
\label{fig:stdknet}
\end{figure}

\item[C3] Around each joint $z_\mu$ there is neighborhood in which $\CW$ is equivalent to one of the local models in \cref{fig:junction}.

\begin{figure}
\begin{tikzpicture}
\newcommand*{\farr}{.3};
\newcommand*{\sarr}{.76};
\node [matrix] at (0,0) {
\foreach \s/\t/\l in {{0,2}/{3,0}/{ij},{0,0}/{3,2}/{kl}} {\draw [thick,decoration={markings, mark=at position \farr with \arrow{stealth'}, mark=at position \sarr with \arrow{stealth'}},postaction=decorate] (\s) -- (\t) node[pos=.1,above]{\textsf{\l}} node[pos=0.9,above]{\textsf{\l}};};\\% 
};
\node [matrix] at (4.5,0) {
\foreach \s/\t/\l in {{0,2}/{3,0}/{ij},{0,0}/{3,2}/{jk}} {\draw [thick,decoration={markings, mark=at position \farr with \arrow{stealth'}, mark=at position \sarr with \arrow{stealth'}},postaction=decorate] (\s) -- (\t) node[pos=.1,above]{\textsf{\l}} node[pos=0.9,above]{\textsf{\l}};};
\draw [thick,decoration={markings, mark=at position .5 with \arrow{stealth'}},postaction=decorate] (1.5,1) -- (3,1) node[pos=0.9,above]{\textsf{ik}};
\\
};
\node [matrix] at (9,0) {
\foreach \s/\t/\l in {{0,2}/{3,0}/{ij},{0,1}/{3,1}/{ik},{0,0}/{3,2}/{jk}} {\draw [thick,decoration={markings, mark=at position \farr with \arrow{stealth'}, mark=at position \sarr with \arrow{stealth'}},postaction=decorate] (\s) -- (\t) node[pos=.1,above]{\textsf{\l}} node[pos=0.9,above]{\textsf{\l}};};\\
};
\end{tikzpicture}
\caption{The local models of a junction in a spectral network.}
\label{fig:junction}
\end{figure}
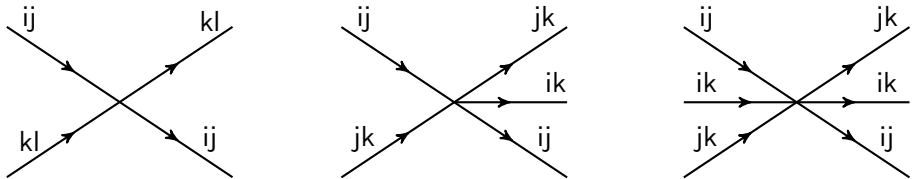

\item[C4] If an oriented segment with label $(o,ij)$ ends at a singular point $\mathfrak{s}_n$,
then $i$ and $j$ lie in the ordered subset $o(\mathfrak{s}_n)$,
and with respect to the ordering of $o(\mathfrak{s}_n)$ we have $i < j$.
\end{description}

The local models referred to above for nonsimple branch points are as follows.  Here we write $\om  = e^\frac{2\pi i}{N}$ and $\om_{ij} = \om^i - \om^j$.

\begin{defn}\label{def:stdknet}
Fix $N>2$.  The spectral network $\stdknet$ is subordinate to the branched cover
\[
\Si_N \coloneqq \{ \la \in T^*\BC: \la^N - z dz^N =0\}
\]
of $\BC$.  Its walls lie along the rays $\{te^{i\vphi}:t \geq 0\}$ such that
\[
2(N+1)(\vphi - \frac\pi2)\equiv0 \text{ mod } 2\pi.
\]
We label the sheets of $\Si_N$ by $\{0,1,\dotsc,N-1\}$ by choosing a generic branch cut from $0$ to $\infty$ and letting $z^{1/N}$ refer to the branch which takes positive real values on $\BR_+$.  The $k$th sheet of $\Si_N$ corresponds by the differential
\[
\la_k = \om^k z^{1/N} dz.
\]
There is an outwardly oriented $ij$-wall lying along the ray of phase $\vphi$ if and only if $\vphi = \frac{-N}{N+1}\arg \om_{ij}$, where the $(N+1)$st root is computed in the sector away from the branch cut.
\end{defn}

In general any wall of $\stdknet$ will be coincident with several other walls with distinct labels.  It is sometimes convenient to refer to a maximal collection of coincident walls as a multiwall.  Then for $N>2$ the network $\stdknet$ has $N^2-1$ walls lying along $2(N+1)$ multiwalls, the latter being in correspondence with the values taken by $\arg(\om_{ij})$ for $0 \leq i,j \leq N-1$.  

The natural examples of spectral networks are associated with branched covers $\Si \subset T^*C$ of a punctured Riemann surface $C$ embedded inside its cotangent bundle.  For our purposes, it is sufficient in the following definition to assume $\th$ is generic and omit a detailed discussion of the behavior of $\Si$ near the punctures of $C$.  

\begin{defn}
For a generic phase $\th$, $\CW_{\Si,\th}$ is the minimal spectral network subordinate to $\Si$ whose walls have the following property.  If $p_c$ carries a local label $(o,ij)$ and $\la^{(i)}$, $\la^{(j)}$ are the 1-forms associated with sheets $i$ and $j$, then
\[
\arg \<p_c'(t)|\la^{(i)}-\la^{(j)}\> = \th
\]
for any oriented parametrization of $p_c$.  
\end{defn}

In particular, the walls of $\CW_{\Si,\th}$ are BPS strings in the sense of \cref{def:bpsstring}.  In this notation $\stdknet$ is the spectral network $\CW_{\Si_N,0}$ of phase $0$ subordinate to $\Si_N$.  

\subsection*{Nonabelianization and Path-Lifting}
A spectral network subordinate to $\Si \onto C$ provides a mechanism for producing twisted $GL_N(\BC)$-local systems on $C$ from twisted $GL_1(\BC)$-local systems on $\Si$ (from now on, we simply write $GL_k$ for $GL_k(\BC)$).  Recall that a twisted $GL_N$-local system on a surface $S$ is a rank $N$ vector bundle $E$ on the unit tangent bundle $\ti S$ of $S$ and a parallel transport on $E$ such that the holonomy around any fiber is $-1$.  Any smooth path in $S$ has a canonical lift to $\ti S$, and in practice we often think of a twisted local system as assigning parallel transports to smooth paths in $S$ via their canonical lifts.  

Let $\CM_{GL_N}^{tw}(C)$ and $\CM_{GL_1}^{tw}(\Si)$ be the moduli spaces of twisted local systems on $C$ and $\Si$, the former with possibly irregular singularities at the singular points of $C$.  To a spectral network $\CW$ subordinate to $\Si$ is associated a nonabelianization map
\[
\Psi_\CW: \CM_{GL_1}^{tw}(\Si) \to \CM_{GL_N}^{tw}(C).
\]
In the body of the paper we do not require a careful discussion of the singularities of flat $GL_N$-connections, so we omit the rather lengthy digression this would require.  All that will be important for us is that if is a smooth closed curve in $C$ and $V$ a representation of $GL_N$, we have an unambiguous definition of the pullback
\[
\Psi_{\CW}^* \tr_V\Hol_{\ti\wp} \in \BC[ \CM_{GL_1}^{tw}(\Si)].
\]

The definition of $\Psi_\CW$ uses extra data, soliton sets, associated to the walls $p_c$ of $\CW$.  For $z \in p_c$, a soliton $s(z)$ is an immersion of $[0,1]$ into $\Si$, considered up to regular homotopy, which begins and ends on preimages of $z$.  If $p_c$ carries the label $(o, ij)$, $s(z)$ is compatible with $o$ if it begins on $z^{(i)}$ and ends on $z^{(j)}$ (where $z^{(i)}$ and $z^{(j)}$ are the lifts of $z$ to sheets $i$ and $j$), and its projection to $C$ begins following $p_c$ with orientation $-o$ and ends following $p_c$ with orientation $o$.  For each point $z \in p_c$ and orientation $o$ of $p_c$, the soliton set $\CS^o_c(z)$ is a collection of solitons compatible with $o$.  For a generic $\CW$ it is completely determined by the following conditions:

\begin{description}
\item[ST1] The soliton sets $\CS_c^o(z)$ vary continuously (in the obvious sense) as $z$ varies continuously in $p_c$.  Thus we often write soliton sets simply as $\CS_c^o$.
\item[ST2] If one endpoint of $p_c$ lies on a branch point, its soliton set consists exactly of the following ``light'' soliton $s_c$.  If $p_c$ is labeled $(o,ij)$ with $o$ directed away from $\fb$, $s_c(z)$ is the immersion that begins at $z^{(i)}$, ends at $z^{(j)}$, and whose projection to $C$ lies on $p_c$ except in a small neighborhood of $\fb$.  
\item[ST3] The soliton sets of the walls meeting at a junction $z_\mu$ are related in a standard way.  We refer to \cite[Section 9.3]{Gaiotto} for the detailed rule, as in the body of the paper we only require computations involving \textbf{ST2}.
\end{description}

In general the above rules uniquely determine the soliton sets of $\CW$, for example in the case of a spectral network of a generic phase $\th$ subordinate to some $\Si \subset T^*C$.

The spectral network $\CW$ assigns to each open path $\ti \wp$ in the unit tangent bundle $\ti C$ of $C$ an expression
\[
\bF(\ti \wp,\CW) = \sum_{\ba} \fro'(\ti \wp,\CW,\ba) \bX_\ba.
\]
Here $\bX_\ba$ is a formal variable associated with a homotopy class $\ba$ of open path in $\ti \Si$.  These are considered modulo the relations that $\bX_{\ba \ba'} = \bX_\ba \bX_\ba'$ if $\ba$ and $\ba'$ are concatenable, and that if $\ba$ and $\ba'$ project to the same homotopy class in $\Si$,
\[
\bX_\ba/\bX_{\ba'} = (-1)^{w(\ba,\ba')}.
\]
Here $w(\ba,\ba')$ is the winding number of $\ba^{-1}\ba'$ around the fiber.  The $\fro'(\ti \wp,\CW,\ba)$ are integers uniquely determined by the following properties:
\begin{description}
\item[PL1] If $\ti\wp$, $\ti\wp'$ are concatenable in $\tilde C$,
\begin{equation} \label{eq:composition-bold}
\bF(\ti\wp \ti\wp', \CW) = \bF(\ti\wp, \CW) \bF(\ti\wp', \CW).
\end{equation}
\item[PL2] Let 
\[
\bD(\ti\wp) = \sum_{i=1}^N \bX_{\ti\wp^{(i)}},
\]
where $\ti\wp^{(i)}$ is the canonical lift of $\ti \wp$ to the $i$th sheet of $\wt\Si$.  If $\ti \wp$ does not intersect the preimage of $\CW$ in $\ti C$, then
\[
\bF(\ti \wp,\CW) = \bD(\ti\wp).
\]
\item[PL3] Suppose $\ti \wp$ intersects $\ti \CW$ at exactly one point $\ti z$, whose projection to $C$ we denote by $z$.  If the intersection is not transverse we perturb it to be so.  Let $\ti \wp_+$ and $\ti \wp_-$ be the parts of $\ti \wp$ before and after this intersection, so $\ti \wp = \ti \wp_+ \ti \wp_-$.  Then
\begin{equation} \label{eq:F-one-cross}
 \bF(\ti\wp, \tilde\CW) = \bD(\ti\wp_+) \left(1 + \sum_{s_\nu \in \CS_c(z)} \bX_{\ba(s_\nu,\ti\wp)}\right) \bD(\ti\wp_-),
\end{equation}
where $\ba(s_\nu,\ti\wp)$ is defined as follows.  Each soliton $s_\nu$ has a canonical lift $\ti s_\nu$ to $\ti \Si$.  Let $U_z$ be the unit tangent space at $z$ and $o_+,o_-\in U_z$ the intial and final unit tangent vectors to the projection of $s_\nu$ to $C$.  Thus $o_+$ and $o_-$ point in opposite directions along a wall of $\CW$, and $\ti s_\nu$ starts and ends at their canonical lifts $o_+^{(i)}, o_-^{(j)}$ to $U_{z^{(i)}}, U_{z^{(j)}}$.  Let $A_+$ and $A_-$ be paths in $U_z$ from $z$ to $o_+$ and from $o_-$ to $z$, such that their composition is homotopic to a simple arc around the half of $U_z$ containing $\ti z$.  Then 
\[
 \ba(\nu,\ti\wp) = A^{(i)}_+ \tilde s_\nu(z) A^{(j)}_-.
\]
\end{description}

We refer to the assignment
\[
\ti \wp \mapsto \bF(\ti \wp,\CW)
\]
as the path-lifting rule defined by $\CW$.  Its essential property is that if $\ti\wp$ and $\ti\wp'$ are two paths on $\tilde C$ that project to the same homotopy class on $C$,
\begin{equation} \label{eq:twisted-homotopy}
\bF(\ti\wp, \CW) = (-1)^{w(\ti\wp, \ti\wp')} \bF(\ti\wp', \CW).
\end{equation}
This is proved for simple spectral networks in \cite{Gaiotto}, and the general case follows from \cref{cor:simpleres}.

Let $(\CL,\nab^\ab)$ be a twisted $GL_1$-local system on $\Si$.  We write the parallel transport of $\nab^\ab$ along $\ba$ as
\[
\bX_\ba(\nab^\ab) \in \Hom (\CL_{\ba(0)},\CL_{\ba(1)}).
\]
The twisted $GL_N$-local system $(E,\nab) = \Psi_\CW(\CL,\nab^\ab)$ on $C$ is defined as follows.  The vector bundle $E$ is the pushforward of $\CL$ to $\wt C$, so the fiber over $\ti z \in \ti C'$ is
\[
E_{\ti z} = \bigoplus_{i=1}^N \CL_{\ti z^{(i)}}.
\]
If $\ti \wp$ is a path with endpoints $\ti z_1, \ti z_2 \in \ti C'_\CW \coloneqq \ti C' \smallsetminus \ti \CW$, then the parallel transport of $\nab$ along $\ti \wp$ is
\[
\bF(\ti \wp,\CW,\nab^\ab) \coloneqq \sum_{\ba} \fro'(\ti \wp,\CW,\ba) \bX_\ba(\nab^\ab) \in \Hom(E_{\ti z_1},E_{\ti z_2}).
\]
It is demonstrated in \cite{Gaiotto} that this defines the parallel transport of a twisted local system on $C$.

\bibliography{TodaCMP}
\bibliographystyle{alpha}

\end{document}